\documentclass[10pt,reqno]{amsart}
\usepackage{amsmath}
\usepackage{amssymb,latexsym}
\usepackage{amscd}
\usepackage{verbatim}
\usepackage[english]{babel}
\usepackage[pdftex,bookmarks=true]{hyperref}
\usepackage[usenames,dvipsnames]{color}

\newtheorem{theorem}{Theorem}[section]

\newtheorem{lemma}[theorem]{Lemma}
\newtheorem{proposition}[theorem]{Proposition}

\newtheorem{thmx}{Theorem}

\newtheorem{definition}[theorem]{Definition}
\newtheorem{remark}[theorem]{Remark}

\numberwithin{equation}{section}

\newcommand{\be}{\begin{equation}}
\newcommand{\bel}[1]{\begin{equation}\label{#1}}
\newcommand{\ee}{\end{equation}}

\newcommand{\barr}{\begin{eqnarray}}
\newcommand{\earr}{\end{eqnarray}}
\newcommand{\bars}{\begin{eqnarray*}}
\newcommand{\ears}{\end{eqnarray*}}

\newtheorem{subn}{\name}

\newcommand{\bsn}[1]{\def\name{#1}\begin{subn}}
\newcommand{\esn}{\end{subn}}

\newtheorem{sub}{\name}[section]

\newcommand{\bs}{\begin{sub}}
\newcommand{\es}{\end{sub}}
\newcommand{\bsl}[1]{\begin{sub}\label{#1}}

\newcommand{\bth}[1]{\def\name{Theorem}
\begin{sub}\label{t:#1}}
\newcommand{\blemma}[1]{\def\name{Lemma}
\begin{sub}\label{l:#1}}
\newcommand{\bcor}[1]{\def\name{Corollary}
\begin{sub}\label{c:#1}}
\newcommand{\bdef}[1]{\def\name{Definition}
\begin{sub}\label{d:#1}}
\newcommand{\bprop}[1]{\def\name{Proposition}
\begin{sub}\label{p:#1}}

\newcommand{\R}{\eqref}

\newcommand{\BA}{\begin{array}}
\newcommand{\EA}{\end{array}}
\newcommand{\BAN}{\renewcommand{\arraystretch}{1.2}
\setlength{\arraycolsep}{2pt}\begin{array}}
\newcommand{\BAV}[2]{\renewcommand{\arraystretch}{#1}
\setlength{\arraycolsep}{#2}\begin{array}}
\newcommand{\BSA}{\begin{subarray}}
\newcommand{\ESA}{\end{subarray}}

\newcommand{\BAL}{\begin{aligned}}
\newcommand{\EAL}{\end{aligned}}
\newcommand{\BALG}{\begin{alignat}}
\newcommand{\EALG}{\end{alignat}}
\newcommand{\BALGN}{\begin{alignat*}}
\newcommand{\EALGN}{\end{alignat*}}
\def\angb<#1>{\langle #1 \rangle}

\def\ga{\alpha}

\def\CS{{\mathcal S}}

\def\CT{{\mathcal T}}

\def\Z{\mathbb{Z}}

\def\R{\mathbb{R}}

\let\ol=\overline

\let\.=\cdot
\let\0=\emptyset

\def\O{\Omega}
\def\M{\mathcal{M}}

\def\R{\mathbb{R}}

\def\Z{\mathbb{Z}}

\def\XX{\mathcal{X}}

\def\bs{\backslash}
\def\ol{\overline}

\def\stst{\subset\subset}

\def\al{\alpha}
\def\be{\beta}
\def\ep{\epsilon}

\def\la{\lambda}

\def\si{\sigma}

\def\Om{\Omega}
\def\de{\delta}
\def\De{\Delta}
\def\ga{\gamma}

\begin{document}

\title[Nonlocal dispersal equations in time-periodic media]{Nonlocal dispersal equations in time-periodic media: principal spectral theory, bifurcation and asymptotic behaviors}
\author{Hoang-Hung Vo}
\address{Department of Mathematics and Computer Science, Ho Chi Minh City University of Science, Vietnam National University, 227 Nguyen Van Cu, District 5, Ho Chi Minh city, Vietnam}
\email{vhhung@hcmus.edu.vn}
\thanks{H.-H. Vo is supported by the Vietnam National Foundation of Scientific and Technology Development (NAFOSTED)-Project 101.02-2016.26.}

\author{Zhongwei Shen$^{*}$}
\address{Department of Mathematical and Statistical Sciences, University of Alberta, 632 Central Academic Building,
Edmonton, AB T6G 2G1, Canada}
\email{zhongwei@ualberta.ca}
\thanks{$^*$ Corresponding author}


\begin{abstract} 
This paper is devoted to the investigation of the following nonlocal dispersal equation
$$
u_{t}(t,x)=\frac{D}{\sigma^m}\left[\int_{\Om}J_\sigma(x-y)u(t,y)dy-u(t,x)\right]+f(t,x,u(t,x)), \quad t>0,\quad x\in\ol{\Om},
$$
where $\O\subset\R^{N}$ is a bounded and connected domain with smooth boundary, $m\in[0,2)$, $D>0$ is the dispersal rate, $\si>0$ characterizes the dispersal range, $J_{\si}=\frac{1}{\si^{N}} J\left(\frac{\cdot}{\sigma}\right)$ is the scaled dispersal kernel, and $f$ is a time-periodic nonlinear function of generalized KPP type. We first study the principal spectral theory of the linear operator associated to the linearization of the equation at $u\equiv0$. We obtain an easily verifiable and general condition for the existence of the principal eigenvalue as well as important sup-inf characterizations for the principal eigenvalue. We next study the influence of the principal eigenvalue on the global dynamics and confirm the criticality of the principal eigenvalue being zero. It is then followed by the study of the effects of the dispersal rate $D$ and the dispersal range characterized by $\si$ on the principal eigenvalue and the positive time-periodic solution, and prove various asymptotic behaviors of the principal eigenvalue and the positive time-periodic solution when $D,\si\to0^{+}$ or $\infty$. Finally, we establish the  maximum principle for time-periodic nonlocal operator.


\end{abstract}

\subjclass[2010]{Primary 35B50, 47G20; secondary 35J60}



\keywords{Nonlocal dispersal equation, principal eigenvalue, generalized principal eigenvalue, bifurcation, positive solution, global dynamics, maximum principle}

\maketitle

\tableofcontents


\section{Introduction and main results}

The present paper is devoted to the investigation of the following nonlocal dispersal equation (or, integro-differential equation) in spatio-temporal heterogeneous environments
\begin{equation}\label{main-eqn}
u_{t}(t,x)=\frac{D}{\sigma^m}\left[\int_{\Om}J_\sigma(x-y)u(t,y)dy-u(t,x)\right]+f(t,x,u(t,x)),\quad t>0,\quad x\in\ol{\Om},
\end{equation}
where $\O\subset\R^{N}$ is a bounded and connected domain with smooth boundary, $m\in[0,2)$, $D>0$, $\si>0$ and $J_{\si}(x)=\si^{-N} J\left(x/\sigma\right)$ for $x\in\R^{N}$. The operator 
$$
u\mapsto\frac{D}{\sigma^m}\left[\int_{\Om}J_\sigma(\cdot-y)u(y)dy-u\right]
$$
is called the nonlocal dispersal operator. The dispersal kernel $J$ and the nonlinearity  $f(t,x,s)$ satisfy the following assumptions:

\begin{itemize}
\item[\rm\bf(H1)] $J\in C(\R^{N})$ is nonnegative, symmetric and supported in $B_\gamma(0)$ for some $\gamma>0$, and satisfies $J(0)>0$ and $\int_{\R^{N}}J(x)dx=1$, where $B_\gamma(0)\subset\R^{N}$ is the open ball centered at $0$ with radius $\ga$.

\item[\rm\bf(H2)] $f:\R\times\ol{\Om}\times\R\mapsto\R$ is of KPP type and satisfies the following conditions.

\begin{enumerate}
\item $f(\cdot,x,s)\in C(\R)$, $f(t,\cdot,s)\in C^1(\ol{\Om})$ and $f(t,x,\cdot)\in C^1(\R)$.

\item $f(t,x,0)=0$ for all $(t,x)\in\R\times\ol{\Om}$ and there is $T>0$ such that 
$$
f(t+T,x,s)=f(t,x,s),\quad\forall(t,x,s)\in\R\times\ol{\Om}\times\R.
$$

\item For all $(t,x)\in\R\times\ol{\Om}$, the function $s\mapsto\frac{f(t,x,s)}{s}$ is decreasing on $(0,\infty)$. 

\item There exists $S\in C(\R\times\ol{\Om})\cap L^{\infty}(\R\times\ol{\Om})$ such that 
$$
f(t,x,S(t,x))\leq0,\quad\forall(t,x)\in\R\times\ol{\Om}.
$$
\end{enumerate}
\end{itemize}

The equation \eqref{main-eqn} is often used to model the evolution of species that exhibit long range internal interactions and are subject to seasonal effects and spatial variations (see e.g. \cite{BCV1,BCV2,JL1,JL2,HMP01,HMMV03,PMJ, ShXi15-2}). In this context, whether the species can survive or not, and the eventual distributions of the species if survive are fundamental issues. In terms of the equation \eqref{main-eqn}, these issues are closely related to the global dynamics of the solutions of \eqref{main-eqn} and corresponding effects of the \textit{dispersal rate} and the \textit{dispersal range} characterized $D$ and $\si$, respectively, on the solutions. The number $m$ is referred to as the \textit{cost parameter} (see e.g. \cite{BCV1,BCV2,HMP01,HMMV03,ShXi15-2}).

As it is known from \cite{BCV1,BCV2,RS12, ShXi15-2} and references therein, the principal spectral theory of the linear operator associated to the equation \eqref{main-eqn} linearized at zero, namely, the nonlocal parabolic-type operator 
$$
v\mapsto-v_{t}(t,x)+\frac{D}{\si^{m}}\left[\int_{\Om}J_{\si}(x-y)v(t,y)dy-v(t,x)\right]+f_{s}(t,x,0)v(t,x),
$$
plays an essential role in investigating the equation \eqref{main-eqn}. To study the principal spectral theory of the above operator, it is natural to consider the operator with general continuous $T$-periodic coefficient $a(t,x)$  as follows
\begin{equation}\label{main-eqn-linear}
L_{\Om}[v](t,x)=-v_{t}(t,x)+D\left[\int_{\Om}J(x-y)v(t,y)dy-v(t,x)\right]+a(t,x)v(t,x),\quad (t,x)\in\R\times\ol{\Om},
\end{equation}
where $a\in C(\R\times\ol{\O})$ satisfies
$$
a(t+T,x)=a(t,x),\quad (t,x)\in\R\times\ol{\Om}. 
$$
We define the following spaces $\XX_{\Om}$, $\XX_{\Om}^{+}$ and $\XX_{\Om}^{++}$ :
\begin{equation}\label{T-periodic-spaces}
\begin{split}
\XX_{\Om}&=\big\{v\in C^{1,0}(\R\times{\ol\Om})\big|v(t+T,x)=v(t,x),\,\,(t,x)\in\R\times{\ol\Om}\big\},\\
\XX_{\Om}^{+}&=\big\{v\in\XX_{\Om}\big|v(t,x)\geq0,\,\,(t,x)\in\R\times{\ol\Om}\big\},\quad\text{and}\\
\XX_{\Om}^{++}&=\big\{v\in\XX_{\Om}\big|v(t,x)>0,\,\,(t,x)\in\R\times{\ol\Om}\big\},
\end{split}
\end{equation}
where $C^{1,0}(\R\times\ol{\O})$ denotes the class of functions that are $C^1$ in $t$ and continuous in $x$. Set
\begin{equation*}
 a_{T}(x):=\frac{1}{T}\int_{0}^{T}a(t,x)dt,\quad x\in\ol{\Om}.
\end{equation*}

The principal spectral theory for nonlocal elliptic-type operators and their properties have been extensively investigated in \cite{Cov10,Coville1,Coville2,BCV1,ShXi15-1} and references therein. In particular, Coville et al. proved in \cite{Coville1,Coville2, Cov10} a sharp sufficient condition for the existence of the principal eigenvalue using the generalized principal spectral theory developed in \cite{BNV94}, while Shen and Xie proved in \cite{ShXi15-1} a necessary and sufficient spectral condition for the existence of the principal eigenvalue using a dynamical system approach. 

The principal spectral theory of nonlocal parabolic-type operators like $L_{\Om}$ is later studied by Rawal and Shen \cite{RS12}. Due to the non-compactness of nonlocal operators and their resolvents, principal eigenvalues do not exist in general. The notion \textit{principal spectrum point} (see Definition \ref{defn-principal}), in replace of principal eigenvalue, was used in \cite{RS12}. Moreover, they proved a necessary and sufficient spectral condition for the principal spectrum point becoming the principal eigenvalue. More precisely, they proved the following theorem.

\begin{theorem}[\cite{RS12}]\label{prop-principal-e}
Suppose {\bf(H1)} and  {\bf(H2)}. $\la_{1}(-L_{\Om})$ is the principal eigenvalue of $-L_{\Om}$ if and only if 
\begin{equation*}\label{lambda-*}
\la_{1}(-L_{\Om})<\la_{*}:=\min_{x\in\ol{\Om}}\left[D-a_{T}(x)\right],
\end{equation*}
where $\lambda_1(-L_\O)$ is the principal spectrum point of $-L_\O$ (see Definition \ref{defn-principal}). Moreover, when $\la_{1}(-L_{\Om})$ is the principal eigenvalue of $-L_{\Om}$, it is geometrically simple and has an eigenfunction in $\XX_{\Om}^{++}$. 
\end{theorem}

It is nice that the condition $\la_{1}(-L_{\Om})<\la_{*}$ in the above theorem is necessary and sufficient. However, it turns out it is rather hard to check when this condition is true, since it is related to both $\la_{1}(-L_{\Om})$ and $a(t,x)$. Therefore, it is necessary to find an easily verifiable sufficient condition for $\la_{1}(-L_{\Om})$ being the principal eigenvalue of $-L_{\Om}$. This leads to our first main result. Besides this, we also prove sup-inf characterizations of the principal eigenvalue under this condition. These results are summarized in the following theorem.

\begin{thmx}[Principal eigenvalue and sup-inf characterizations]\label{thm-pe-introduction}
Suppose {\bf(H1)} and {\bf(H2)}. If 
\begin{equation}\label{newcond1}
\frac{1}{\max_{y\in\ol{\Om}}a_{T}(y)-a_{T}}\notin L^{1}_{loc}(\ol{\Om}),
\end{equation}
then $\la_{1}(-L_{\Om})$ is the principal eigenvalue of $-L_{\Om}$. Moreover, there holds
$$
\la_{1}(-L_{\Om})=\lambda_{p}(-L_{\Om})=\lambda_{p}'(-L_{\Om}),
$$
where 
\begin{equation}\label{characterization}
\begin{split}
\lambda_{p}(-L_{\Om}):&=\sup\left\{\lambda\in\R:\exists\phi\in \mathcal{X}^{++}_{\Om}\,\,\text{s.t.}\,\,(L_{\Om}+\la)[\phi]\leq0\,\,\text{in ${\R\times\ol\O}$}\right\},\\
\lambda_{p}'(-L_{\Om}):&=\inf\left\{\lambda\in\R:\exists\phi\in \mathcal{X}^{++}_{\Om}\,\,\text{s.t.}\,\,(L_{\Om}+\la)[\phi]\geq0\,\,\text{in ${\R\times\ol\O}$}\right\}.
\end{split}
\end{equation}

\end{thmx}

One  sees that the condition \eqref{newcond1} concerns the smoothness of $a_{T}(x)$ near its maximum points. More importantly, \eqref{newcond1} is not related to the dispersal kernel $J$ and $\la_{1}(-L_{\Om})$, and  therefore, it is very useful when we study the equation \eqref{main-eqn}  with scaled kernels later. Although the condition \eqref{newcond1} is only a sufficient condition, it is indeed sharp in the sense that a function $a(t,x)$ unfulfilling  \eqref{newcond1} can be constructed so that $-L_{\Om}$ does not admit any eigenvalue  (see e.g. \cite{Cov10}). The quantities $\la_{p}(-L_{\Om})$ and $\la_{p}'(-L_{\Om})$, always well-defined, are usually called  the \textit{generalized principal eigenvalues} of $-L_{\Om}$. These notions are originally introduced by Berestycki, Nirenberg and Varadhan in \cite{BNV94} to study the principal spectral theory of elliptic operators. Since then, they are widely used to study the principal spectral theory of various linear operators associated to reaction-diffusion equations and nonlocal dispersal equations (see \cite{BCV1,BR1,BR2,Cov10,Nadin1,Vo1,Vo2} and references therein). The equivalence of $\la_{1}(-L_{\Om})$, $\la_{p}(-L_{\Om})$ and $\la_{p}'(-L_{\Om})$ under the condition \eqref{newcond1} provides not only $\sup$-$\inf$ characterizations of $\la_{1}(-L_{\Om})$, but also alternative and powerful tools in the spirit of analysis to  study deeper qualitative properties of $\la_{1}(-L_{\Om})$ to be presented. 


In the presence of the principal spectral theory, namely, Theorem \ref{thm-pe-introduction}, we then move forward to study the global dynamics of solutions of the equation \eqref{main-eqn} in the non-scaled case with $m=0$ and $\si=1$, that is,
\begin{equation}\label{main-eqn-Dirichlet-simple1}
u_{t}(t,x)=D\left[\int_{\Om}J(x-y)u(t,y)dy-u(t,x)\right]+f(t,x,u(t,x)),\quad t>0,\quad x\in\ol{\Om}.
\end{equation}
To do so, we need to investigate Liouville-type results, namely, the existence/nonexistence, of positive entire solutions of the equation
\begin{equation}\label{main-eqn-Dirichlet-simple}
u_{t}(t,x)=D\left[\int_{\Om}J(x-y)u(t,y)dy-u(t,x)\right]+f(t,x,u(t,x)),\quad t\in\R,\quad x\in\ol{\Om}.
\end{equation}
Before stating out results, we remark that the global dynamics of \eqref{main-eqn-Dirichlet-simple1}  has been partially investigated by Rawal and Shen in \cite{RS12}. They proved that if $\la_{1}(-L_{\Om})<0$, solutions of \eqref{main-eqn-Dirichlet-simple1} converges, as $t\to\infty$, to the unique positive $T$-periodic solution of \eqref{main-eqn-Dirichlet-simple}. But, the global dynamics of \eqref{main-eqn-Dirichlet-simple1} and Liouville-type result of \eqref{main-eqn-Dirichlet-simple} when $\la_{1}(-L_{\Om})\geq0$ has not been investigated yet. Moreover, the problem in the critical case $\lambda_1(-L_\O) = 0$ seems to be challenging. These issues will be studied in the present paper.


Let $a(t,x)=f_{s}(t,x,0)$ for $(t,x)\in\R\times\ol{\Om}$, and $L_\O$ be the linear operator associated to the linearization of \eqref{main-eqn-Dirichlet-simple} at $u\equiv0$ defined by (\ref{main-eqn-linear}).  We prove the following theorem. 

\begin{thmx}[Bifurcation and global dynamics]\label{thm-persistence-criterion-b}
Suppose {\bf(H1)}, {\bf(H2)} and \eqref{newcond1}. Then, the equation \eqref{main-eqn-Dirichlet-simple} admits a solution in the space $\XX_{\Om}^{++}$ if and only if $\la_{1}(-L_{\Om})<0$. Moreover, if exists,  the solution denoted by $u^{*}$ is unique  in $\XX_{\Om}^{++}$.

%

Also, let $u(t,x;u_{0})$ be a solution of \eqref{main-eqn-Dirichlet-simple1} with initial data $u_0\in C(\ol{\Om})$, which is not identically zero and non-negative, then the following statements hold.
\begin{enumerate}
\item[(i)]
If $\la_{1}(-L_{\Om})<0$, then 
\begin{equation*}
\|u(t,\cdot;u_{0})-u^{*}(t,\cdot)\|_{\infty}\to0\quad\text{as}\quad t\to\infty,
\end{equation*}
where $\|\cdot\|_{\infty}$ is the sup norm on $C(\ol{\Om})$;

\item[(ii)] If $\la_{1}(-L_{\Om})>0$, then 
\begin{equation*}
\|u(t,\cdot;u_{0})\|_{\infty}\to0\quad\text{as}\quad t\to\infty.
\end{equation*}

\item[(iii)] If $\la_{1}(-L_{\Om})=0$ and $u(t,x;u_{0})$ is equi-continuous in space, namely, there exists $t_{0}>0$ such that for any $\ep>0$ there exists $\de>0$ such that 
\begin{equation}\label{equi}
x,y\in\ol{\Om},\,\,|x-y|\leq\de\quad\text{implies}\quad|u(t,x)-u(t,y)|\leq\de,\,\,\forall t\geq t_{0},
\end{equation}
then 
\begin{equation*}
\|u(t,\cdot;u_{0})\|_{\infty}\to0\quad\text{as}\quad t\to\infty.
\end{equation*}

\end{enumerate}
\end{thmx}

In Theorem \ref{thm-persistence-criterion-b}, we fully characterize the bifurcation of non-negative $T$-periodic solutions of the equation \eqref{main-eqn-Dirichlet-simple} by the sign of the principal eigenvalue. Meanwhile, the global dynamics of \eqref{main-eqn-Dirichlet-simple1} in the case $\la_{1}(-L_{\Om})\geq0$ is studied. It is worthwhile to point out that  in the case $f(t,x,s)=f(x,s)$, the global dynamics of \eqref{main-eqn-Dirichlet-simple1} based on the bifurcation result of \eqref{main-eqn-Dirichlet-simple} have attracted a lot of attention recently due to their significance in applications and underlying mathematical challenges (see e.g. \cite{Coville1,Cov10,ShXi15-1,BCV1,BCV2,ShZh10,Vo1}). In particular, to investigate the problem, the authors in \cite{Coville1, Cov10, BCV1,BCV2} took a PDE approach, while the authors in \cite{ShXi15-1, ShZh10} employed a dynamical system approach. In the proof of Theorem \ref{thm-persistence-criterion-b}, we take advantage of both PDE and dynamical system approaches to tackle those essential difficulties stemming from the lack of regularizing effects of the semigroup generated by the nonlocal dispersal operator and the presence of time-dependence of $f$. 

In the critical case $\lambda_1(-L_\O)=0$, the global dynamics is only proven under the additional condition \eqref{equi}. We remark that in the case $f(t,x,s)=f(x,s)$ treated in \cite{Coville1, Cov10, BCV1,BCV2}, such a condition is not required thanks to a Harnack-type inequality for nonlocal elliptic-type equations and bootstrap arguments. But, for nonlocal parabolic-type equations as in our case, no Harnack-type inequality is known, and bootstrap arguments together with the variation of constants formula are not helping due to the lack of regularizing effects of the semigroup generated by the nonlocal dispersal operator as just mentioned. Although, we believe that \eqref{equi} is  sharp,  the global dynamics of \eqref{main-eqn-Dirichlet-simple1} is still an open problem without this condition.

\begin{description}
\item[Open problem 1] Can one prove the global dynamics of the equation \eqref{main-eqn-Dirichlet-simple1} in the critical case $\lambda_1(-L_\O)=0$ without the additional condition \eqref{equi}?
\end{description}

Based on Theorem \ref{thm-pe-introduction} and Theorem \ref{thm-persistence-criterion-b}, we now turn to study the effects of the dispersal rate $D$ and the dispersal range characterized by $\si$ on the principal eigenvalue and the positive $T$-periodic solution associated to the equation \eqref{main-eqn}. 

We first study the effects of the dispersal rate $D$. For this purpose, it is more convenient to consider those associated to the non-scaled equations \eqref{main-eqn-Dirichlet-simple1} and \eqref{main-eqn-Dirichlet-simple}. In the next theorem, we write $\la_{1}^{D}(-L_{\Om})$ instead of $\la_{1}(-L_{\Om})$ to indicate the $D$-dependence.

\begin{thmx}[Effects of the dispersal rate]\label{thm-effect-dispersal-rate-intro}
Suppose {\bf(H1)}, {\bf(H2)} and \eqref{newcond1}. The following properties hold.

\begin{enumerate}
\item There holds
$$
\la_{1}^D(-L_{\Om})\to\begin{cases}
-\max_{x\in\ol{\Om}}a_{T}(x) &\quad\text{as}\quad D\to0^{+},\\
\infty&\quad\text{as}\quad D\to\infty.
\end{cases}
$$

\item If $a(t,x)=\al(t)+\beta(x)$, then $D\mapsto\la_{1}^D(-L_{\Om})$ is increasing.

\item If $\max_{x\in\ol{\Om}}a_{T}(x)>0$, then the equation \eqref{main-eqn-Dirichlet-simple} admits a unique solution $u^{*}_{D}\in\XX_{\Om}^{++}$ that is globally asymptotically stable for each $0<D\ll1$. The equation \eqref{main-eqn-Dirichlet-simple} admits no solution in the space $\XX^{++}_{\Om}$ for each $D\gg1$.

\item If $\min_{x\in\ol{\Om}}a_{T}(x)>0$, then there holds the limit
$$
\lim_{D\to0^{+}}u^{*}_{D}(t,x)=v^{*}(t,x)\quad\text{uniformly in}\quad(t,x)\in\R\times\ol{\Om},
$$
where $v^*(t,x)$ is the unique positive and $T$-periodic solution of the equation $v_{t}=f(t,x,v)$ for every $x\in\ol{\Om}$.

\end{enumerate}
\end{thmx}

We emphasize that due to the unboundedness of $L_{\Om}$ and the non-self-adjointness of $L_{\Om}$ resulting in the lack of the usual $L^{2}(\Om)$ variational formula for the principal eigenvalue if exists, we cannot invoke the techniques used in the  papers \cite{BCV1, Cov10,Coville1,ShXi15-1} and this causes lots of troubles. Thanks to the characterizations \eqref{characterization}, these difficulties raised in the proof of Theorem \ref{thm-effect-dispersal-rate-intro} can be overcome. These characterizations are also very useful in the study of the effect of dispersal range in Theorem \ref{thm-scaling-limits-introduction}. Theorem \ref{thm-effect-dispersal-rate-intro}(1)(3) shows that, provided $\max_{x\in\ol{\Om}}a_{T}(x)>0$, the small dispersal rates are favored, while the large dispersal rates are always unfavored. It would be interesting to know whether or not $\la_{1}^{D}(-L_{\Om})$ is monotone with respect to $D$. It is referred to \cite{HMP01} for the construction of a non-monotone sequence of principal eigenvalues of elliptic operators with Neumann boundary condition, and therefore, we believe that there is no monotonicity in general. However, in a special case as in the last statement of Theorem \ref{thm-effect-dispersal-rate-intro}(2), we prove the monotonicity.

Now, we study the effects of the dispersal range characterized by $\si$. To do so, let us consider the following operator
$$
L_{\Om,m,\si}[v](t,x)=-v_{t}(t,x)+\frac{D}{\si^{m}}\left[\int_{\Om}J_{\si}(x-y)v(t,y)dy-v(t,x)\right]+a(t,x)v(t,x),\quad (t,x)\in\R\times\ol{\Om}
$$
associated to the linearization of \eqref{main-eqn} at $u\equiv0$. We prove the following result.

\begin{thmx}[Scaling limits of the principal eigenvalue]\label{thm-scaling-limit-eigenvalue}
Suppose {\bf(H1)}, {\bf(H2)} and \eqref{newcond1}. 
\begin{enumerate}
\item As $\si\to\infty$, there holds
$$
\la_{1}(-L_{\Om,m,\si})\to\begin{cases}
D-\max_{x\in\ol{\Om}}a_{T}(x),\quad&m=0,\\
-\max_{x\in\ol{\Om}}a_{T}(x),\quad&m>0.
\end{cases}
$$

\item  As $\si\to0^{+}$, there holds
$$
\la_{1}(-L_{\Om,m,\si})\to-\max_{x\in\ol{\Om}}a_{T}(x),
 \quad\quad\forall m\in[0,2).
$$

\item In the case $m=0$, if $\O$ contains the origin and $a(t,x)$ is radially symmetric and radially decreasing with respect to $x$, namely, $a(t,x)=a(t,y)$ if $|x|=|y|$ and $a(t,x)\leq a(t,y)$ if $|x|\geq |y|$ for all $t\in \R$, then $\si\mapsto\la_{1}(-L_{\Om,0,\si})$ is non-decreasing.

\end{enumerate}
\end{thmx}

\begin{remark}\label{threshold}
In the case of Theorem \ref{thm-scaling-limit-eigenvalue}(3), if $\max_{x\in\ol{\Om}}a_{T}(x)\in (0,D)$, then the monotonicity and the continuity (see Proposition \ref{prop-principal-e-variation-1}(5)) of $\si\mapsto\la_{1}(-L_{\Om,0,\si})$ together with Theorem \ref{thm-scaling-limit-eigenvalue}(1)(2) imply the existence of a threshold value $\sigma^*>0$ such that $\la_{1}(-L_{\Om,0,\si})<0$ if and only if $\si<\si^{*}$, and hence, \eqref{main-eqn-Dirichlet-simple} admits a unique solution $u^{*}_{\si}\in \XX_{\Om}^{++}$ that is globally asymptotically stable if and only if $\si<\si^{*}$. This $\sigma^*$ is usually referred to as the {\rm critical range for persistence}.
\end{remark}

Results as in Theorem \ref{thm-scaling-limit-eigenvalue}, in the case of nonlocal elliptic-type operators, have been obtained in \cite{BCV1,ShXi15-1}. The fact being lack of the usual $L^{2}(\Om)$ variational characterization for the principal eigenvalue in our case indeed yields big obstacles, especially, in the study of the limit of $\la_{1}(-L_{\Om,m,\si})$ as $\si\to0^{+}$. This is overcome by the $\sup$-$\inf$ characterizations of the principal eigenvalue as in Theorem \ref{thm-pe-introduction} and an involved analysis of decaying rates in terms of $\si$ of various terms  (see the proof of Theorem \ref{thm-scaling-limit-eigenvalue} in Section \ref{sec-effect-dispersal-rate} for more details). It is also generally understood that  the time-dependence of $f(t,x,s)$ largely complicates the behavior of the principal eigenvalue $\la_{1}(-L_{\Om,m,\si})$ in term of various parameters. We remark that $\si\ll1$ and $\si\gg1$ represent two completely different dispersal strategies. The former says that the dispersal is essentially localized, while the later supports the dispersal over a very long distance. It is interesting to see that the behaviors of the principal eigenvalue are intrinsically different in the cases $m=0$ and $m\in(0,2)$. More precisely, in the case $m\in(0,2)$, both small and large dispersal ranges are favoured provided $\max_{x\in\ol{\Om}}a_{T}(x)>0$. The situation in the case $m=0$ is more complicated. If $\max_{x\in\ol{\Om}}a_{T}(x)\in(0,D)$, small dispersal ranges are favoured and large dispersal ranges are unfavored, while if $\max_{x\in\ol{\Om}}a_{T}(x)>D$, both small and large dispersal ranges are favoured. From this, we can see the involved global dynamics of \eqref{main-eqn} with respect to the dispersal range. 

We further investigate the behaviors of the positive $T$-periodic solution of the following equation
\begin{equation}\label{main-eqn-R}
u_{t}(t,x)=\frac{D}{\sigma^m}\left[\int_{\Om}J_\sigma(x-y)u(t,y)dy-u(t,x)\right]+f(t,x,u(t,x)),\quad t\in\R,\quad x\in\ol{\Om}
\end{equation}
in favoured cases and prove, in particular, the following theorem.

\begin{thmx}[Scaling limits of the positive $T$-periodic solution]\label{thm-scaling-limits-introduction}
Suppose {\bf(H1)}, {\bf(H2)} and \eqref{newcond1}. 

\begin{enumerate}
\item For  each $m\in[0,2)$, the following statements hold.
\begin{enumerate}

\item If $\max_{x\in\ol{\Om}}a_{T}(x)>0$, then there exist $0<\si_{1}\ll1$ such that for each $\si\in(0,\si_{1})$, the equation \eqref{main-eqn-R} has a unique positive $T$-periodic solution $u_{\si}^{*}$ that is globally asymptotically stable.

\item If $\min_{x\in\ol{\Om}}a_{T}(x)>0$, there holds the limit
$$
\lim_{\si\to0^{+}}u^{*}_{\si}(t,x)=v^{*}(t,x)\quad\text{uniformly in}\quad(t,x)\in\R\times\ol{\Om},
$$
where $v^*(t,x)$ is the unique positive and $T$-periodic solution of the equation $v_{t}=f(t,x,v)$ for every $x\in\ol{\Om}$.

\end{enumerate}

\item For  each $m>0$, the following statements hold.
\begin{enumerate}

\item If $\max_{x\in\ol{\Om}}a_{T}(x)>0$, then there exist $1\ll\si_{2}<\infty$ such that for each $\si>\si_{2}$, the equation \eqref{main-eqn-R} has a unique positive $T$-periodic solution $u_{\si}^{*}$ that is globally asymptotically stable.

\item If $\min_{x\in\ol{\Om}}a_{T}(x)>0$, there holds the limit
$$
\lim_{\si\to\infty}u^{*}_{\si}(t,x)=v^{*}(t,x)\quad\text{uniformly in}\quad(t,x)\in\R\times\ol{\Om}.
$$
\end{enumerate}

\end{enumerate}
\end{thmx}

\begin{description}
\item[Open problem 2] Assume that $\max_{x\in\O} a_T(x)\in(0,D)$ in the case of Theorem \ref{thm-scaling-limit-eigenvalue}(3). Can one prove the limit
$$
\lim_{\si\to\sigma_*^{-}}u^{*}_{\si}(t,x)=0,\quad\forall (t,x)\in\R\times\ol{\Om},
$$ 
where $\sigma_*$ is given in Remark \ref{threshold} and what happens if $\max_{x\in\O} a_T(x)=D$?
\end{description}

Results as in Theorem \ref{thm-scaling-limits-introduction} have been obtained by Berestycki, Coville and Vo in \cite{BCV2} in the case $f(t,x,s)=f(x,s)$. More precisely, it was proven in  \cite{BCV2} that for $f(x,s)=a(x)s-s^2$ and $m\in[0,2)$, the unique positive solution of the following equation
$$
\frac{1}{\si^{m}}\left[ \int_{\Om}J_\sigma(x-y)u(y)dy-u(x)\right]+a(x) u(x)-u(x)^2=0,\quad x\in\O
$$
converges, as $\sigma\to0^{+}$, to the nonnegative solution of 
$$
u(x)[a(x)-u(x)]=0,\quad x\in\O,
$$
which may \textit{not be  unique} due to the lack of regularity. Very recently, Shen and Xie studied in \cite{ShXi15-2} the case with $m=2$  and $\si\to0^{+}$ and proved that the principal eigenvalue and the positive $T$-periodic solution converge to that of the corresponding reaction-diffusion equation
\begin{equation}\label{main-eqn-rd}
\begin{cases}
u_{t}(t,x)=d\De u(t,x)+f(t,x,u(t,x)),& x\in\Om,\\
u(t,x)=0,& x\in\partial\Om
\end{cases}
\end{equation}
for some appropriate $d>0$. 



Finally, we  establish a maximum principle for the operator $L_{\Om}$ defined in \eqref{main-eqn-linear}, which is  of fundamental importance and independent interest.

\begin{definition}[Maximum principle] \label{defn-mp}
We say that $L_\O$ admits the {\rm maximum principle} if for any function $u\in C^{1,0}([0,T]\times\ol\O)$ satisfying
 \begin{equation}\label{mp-assumption}
 \begin{cases} 
L_\O[u]\leq 0 & \text{in $(0,T]\times\O$},\\
 u\geq0 & \text{on $(0,T]\times\partial\O$},\\
u(0,\cdot)\geq u(T,\cdot) & \text{in $\Om$},
 \end{cases}
\end{equation}
there must hold $u>0$ in $[0,T]\times\O$ unless $u\equiv0$ in $[0,T]\times\O$.
\end{definition}

\begin{thmx}[Maximum principle]\label{thm-mp-introduction}
Suppose {\bf(H1)}, {\bf(H2)} and \eqref{newcond1}. Then, $L_{\Om}$ admits the maximum principle  if and only if $\lambda_1(-L_\O)\geq0$.
\end{thmx}

For the proof of this result, we only need \eqref{newcond1} so that $\la_{1}(-L_{\Om})$ is the principal eigenvalue of $-L_{\Om}$, therefore, it is independent of other results obtained in the present paper. A similar result in the case $f(t,x,s)=f(x,s)$ has been obtained by Coville in \cite{Cov10}.  We point out that, for elliptic and parabolic operators, the maximum principle holds if and only if the principal eigenvalue is positive (see e.g. \cite{BNV94,PZ15}).

To this end, let us mention that the study of \eqref{main-eqn} serves as the first step to the understanding of the global dynamics of the following competition system, which is of great biological and mathematical interest,  proposed in \cite{HMMV03}, 
\begin{equation}\label{com-system-intro}
\begin{cases}
u_{t}(t,x)=\frac{D_{1}}{\sigma_{1}^m}\left[\int_{\Om}J_{\si_{1}}(x-y)u(t,y)dy-u(t,x)\right]+u(t,x)(a(t,x)-u(t,x)-v(t,x)),\quad x\in\ol{\Om},\\
v_{t}(t,x)=\frac{D_{2}}{\sigma_{2}^m}\left[\int_{\Om}J_{\sigma_{2}}(x-y)v(t,y)dy-v(t,x)\right]+v(t,x)(a(t,x)-u(t,x)-v(t,x)),\quad x\in\ol{\Om}.
\end{cases}
\end{equation}
In fact, the investigation of the global dynamics of \eqref{com-system-intro}, which is one of big challenges in the study of reaction-diffusion equations, relies on the detailed stability analysis of semi-trivial states mainly coming from the investigation of \eqref{main-eqn}. We refer the reader to \cite{HMP01,HN} for the treatment of a similar parabolic competition system.

\textbf{Organization of the paper.} The paper is organized as follows. In Section \ref{sec-pe-vc}, we study the existence of the principal eigenvalue of $-L_{\Om}$ as well as its characterizations. In particular, we prove Theorem \ref{thm-pe-introduction}. In Section \ref{sec-bifurcation-global-dynamics}, we study the bifurcation of non-negative $T$-periodic solutions of \eqref{main-eqn-Dirichlet-simple} when the principal eigenvalue $\la_{1}(-L_{\Om})$ crosses $0$, as well as the global dynamics of \eqref{main-eqn-Dirichlet-simple1}. In particular, Theorem \ref{thm-persistence-criterion-b} is proven. In Section \ref{sec-effect-dispersal-rate}, we study the effects of the dispersal rate $D$ on the principal eigenvalue $\la_{1}(-L_{\Om})$ and the positive $T$-periodic solution, and prove Theorem \ref{thm-effect-dispersal-rate-intro}. In Section \ref{sec-scaling-limits}, we study the effects of the dispersal range characterized by $\si$ on the principal eigenvalue and the positive $T$-periodic solution associated to the \eqref{main-eqn} and \eqref{main-eqn-R}. In particular, we prove Theorem \ref{thm-scaling-limits-introduction}. The last section, Section \ref{sec-mp}, is devoted to the proof of the maximum principle stated in Theorem \ref{thm-mp-introduction}.


\section{Principal eigenvalue and sup-inf characterizations}\label{sec-pe-vc}

In this section, we investigate the principal spectral theory of the operator $L_{\Om}$ defined in \eqref{main-eqn-linear} and prove Theorem \ref{thm-pe-introduction}. Let us start with the definition of principal spectrum point and principal eigenvalue. Recall that the spaces $\XX_{\Om}$, $\XX_{\Om}^{+}$ and $\XX_{\Om}^{++}$ are defined in \eqref{T-periodic-spaces}.



\begin{definition}\label{defn-principal} 
The {\rm principal spectrum point} of $-L_{\O}$ is defined by
\begin{equation*}
\la_{1}(-L_{\Om})=\inf\{\Re\la|\la\in\si(-L_{\Om})\},
\end{equation*}
where $\si(-L_{\Om})$ is the spectrum of $-L_{\Om}$. If $\la_{1}(-L_{\Om})$ is an isolated eigenvalue of $-L_{\Om}$ with an eigenfunction in $\XX_{\Om}^{+}$, then it is called the {\rm principal eigenvalue} of $-L_{\Om}$.  
\end{definition}

It is known (see e.g. \cite{Cov10,ShZh10}) that, due to the nonlocality, neither the operator \eqref{main-eqn-linear} nor its resolvent is compact, and therefore, the Krein-Rutmann theory cannot be applied to derive the existence of the principal eigenvalue. As a matter of fact, $-L_{\Om}$ does not admit a principal eigenvalue in general. It is then of vital significance to know when $\la_{1}(-L_{\Om})$ is indeed the principal eigenvalue of $-L_{\Om}$. For nonlocal parabolic-type operator, the first result was obtained by Rawal and Shen in \cite{RS12} (see Theorem \ref{prop-principal-e}). The authors of \cite{RS12} actually also proved that any eigenvalue $\la\in\R$ of $-L_{\Om}$ having an eigenfunction in $\XX_{\Om}^{+}$ coincides with $\la_{1}(-L_{\Om})$, and therefore, must be the principal eigenvalue. Theorem \ref{prop-principal-e} gives a necessary and sufficient spectral condition to determine whether $\la_{1}(-L_{\Om})$ is the principal eigenvalue of $-L_\O$. However, this spectral condition is hard to verify in general as $\la_{1}(-L_{\Om})$ is not computable and can be barely estimated in general. Although some sufficient conditions based on this spectral condition have been obtained in \cite[Theorem B]{RS12}, they are more or less restricted. Therefore, it is eager to find a more verifiable condition for $\la_{1}(-L_{\Om})$ becoming the principal eigenvalue of $-L_{\Om}$. Here, we provide a sufficient condition that only requires mild smoothness of $a_{T}(x)$ near its maximum points. 

\begin{theorem}\label{thm-new-sufficient-cond-b}
Suppose {\bf(H1)} and {\bf(H2)}. If \eqref{newcond1} holds, then $\la_{1}(-L_{\Om})<\la_{*}$. In particular, $\la_{1}(-L_{\Om})$ is the principal eigenvalue of $-L_{\Om}$.
\end{theorem}

We remark that \eqref{newcond1} is independent of the dispersal kernel $J$, and hence, the condition \eqref{newcond1} is very useful later when we study the equation \eqref{main-eqn} with scaled kernels. Theorem \ref{thm-new-sufficient-cond-b} is the first part of Theorem \ref{thm-pe-introduction}.

Before proving Theorem \ref{thm-new-sufficient-cond-b}, let us write $L_{\Om}=H_{\Om}+K_{\Om}$, where
\begin{equation*}
\begin{split}
H_{\Om}[v](t,x)&=-v_{t}(t,x)-Dv(t,x)+a(t,x)v(t,x),\quad (t,x)\in\R\times\ol{\Om}\\
K_{\Om}[v](t,x)&=D\int_{\Om}J(x-y)v(t,y)dy,\quad (t,x)\in\R\times\ol{\Om}
\end{split}
\end{equation*}
and, recall the following results from \cite{RS12}.

\begin{proposition}\label{prop-rs12-b}
Suppose {\bf(H1)} and {\bf(H2)}.
\begin{enumerate}
\item For any $\al>-\la_{*}$, the inverse $(\al-H_{\Om})^{-1}$ exists. Moreover, there exists $M>0$ such that the estimate
\begin{equation*}
\big((\al-H_{\Om})^{-1}v\big)(t,x)\geq\frac{M}{\al-(-D+a_{T}(x))}v(x),\quad (t,x)\in\R\times\ol{\Om}
\end{equation*}
holds for any $\al\in(-\la_{*},-\la_{*}+1]$ and any $v\in\XX_{\Om}^{+}$ with $v(t,x)\equiv v(x)$.

\item $\la_{1}(-L_{\Om})<\la_{*}$ if and only if there is $\al_{0}>-\la_{*}$ such that $r(K_{\Om}(\al_{0}-H_{\Om})^{-1})>1$, where $r(K_{\Om}(\al_{0}-H_{\Om})^{-1})$ is the spectral radius of $K_{\Om}(\al-H_{\Om})^{-1}$.
\end{enumerate}
\end{proposition}
It is referred to \cite[Proposition 3.5 and Proposition 3.7]{RS12} for the proof of the above proposition.

We now prove Theorem \ref{thm-new-sufficient-cond-b}.

\begin{proof}[{\bf Proof of Theorem \ref{thm-new-sufficient-cond-b}}]
By contradiction, we assume $\la_{1}(-L_{\Om})\geq\la_{*}$. Proposition \ref{prop-rs12-b}(2) yields
\begin{equation}\label{spectral-radii-bound-b}
r(K_{\Om}(\al-H_{\Om})^{-1})\leq1,\quad\forall\al>-\la_{*}.
\end{equation}

It is known from the variation of constants formula that for any $\al>-\la_{*}$ and $v\in\XX_{\Om}$, $(\al-H_{\Om})^{-1}v$ is given by
\begin{equation*}
\big((\al-H_{\Om})^{-1}v\big)(t,x)=\int_{-\infty}^{t}e^{\int_{s}^{t}(-D+a(\tau,x)-\al)d\tau}v(s,x)ds.
\end{equation*}
In particular, there holds the monotonicity of the operator $(\al-H_{\Om})^{-1}$ in the sense that 
$$
\text{$v_{1},v_{2}\in\XX_{\Om}$ with $v_{1}\geq v_{2}$ implies $(\al-H_{\Om})^{-1}v_{1}\geq(\al-H_{\Om})^{-1}v_{2}$.}
$$

Now, Proposition \ref{prop-rs12-b}(1) implies that for each $\al\in(-\la_{*},-\la_{*}+1]$,
\begin{equation*}
\big((\al-H_{\Om})^{-1}1\big)(t,x)\geq\frac{M}{\al-(-D+a_{T}(x))}>0,\quad (t,x)\in\R\times\ol{\Om}.
\end{equation*}
Applying $K_{\Om}$ to both sides of the above estimate, we find
\begin{equation}\label{estimate-0002-b}
\begin{split}
\big(K_{\Om}(\al-H_{\Om})^{-1}1\big)(t,x)&=D\int_{\Om}J(x-y)\big((\al-H_{\Om})^{-1}1\big)(t,y)dy\\
&\geq\int_{\Om}J(x-y)\frac{DM}{\al-(-D+a_{T}(y))}dy,\quad (t,x)\in\R\times\ol{\Om}.
\end{split}
\end{equation}
By the monotonicity of $(\al-H_{\Om})^{-1}$, \eqref{estimate-0002-b} and Proposition \ref{prop-rs12-b}(1), we find for each $(t,x)\in\R\times\ol{\Om}$
\begin{equation*}
\begin{split}
\big((\al-H_{\Om})^{-1}K_{\Om}(\al-H_{\Om})^{-1}1\big)(t,x)&\geq\bigg((\al-H_{\Om})^{-1}\int_{\Om}J(\cdot-y)\frac{DM}{\al-(-D+a_{T}(y))}dy\bigg)(t,x)\\
&\geq\frac{M}{\al-(-D+a_{T}(x))}\int_{\Om}J(x-y)\frac{DM}{\al-(-D+a_{T}(y))}dy,
\end{split}
\end{equation*}
and then
\begin{equation*}
\big((K_{\Om}(\al-H_{\Om})^{-1})^{2}1\big)(t,x)\geq\int_{\Om}J(x-y)\frac{DM}{\al-(-D+a_{T}(y))}\int_{\Om}J(y-z)\frac{DM}{\al-(-D+a_{T}(z))}dzdy.
\end{equation*}

Repeating the above arguments, we find for each $(t,x_{0})\in\R\times\ol{\Om}$ the following estimate
\begin{equation*}
\big((K_{\Om}(\al-H_{\Om})^{-1})^{n}1\big)(t,x_{0})\geq\int_{\Om}\cdots\int_{\Om}\prod_{m=1}^{n}\bigg[J(x_{m-1}-x_{m})\frac{DM}{\al-(-D+a_{T}(x_{m}))}\bigg]dx_{n}\cdots dx_{1}.
\end{equation*}
As a result,
\begin{equation*}
\begin{split}
\|(K_{\Om}(\al-H_{\Om})^{-1})^{n}\|&\geq\max_{(t,x_{0})\in\R\times\ol{\Om}}\big((K_{\Om}(\al-H_{\Om})^{-1})^{n}1\big)(t,x_{0})\\
&\geq\max_{x_{0}\in\ol{\Om}}\int_{\Om}\cdots\int_{\Om}\prod_{m=1}^{n}\bigg[J(x_{m-1}-x_{m})\frac{DM}{\al-(-D+a_{T}(x_{m}))}\bigg]dx_{n}\cdots dx_{1},
\end{split}
\end{equation*}
which implies that for any $x_{0}\in\ol{\Om}$ and $\de>0$,
\begin{equation*}
\begin{split}
&\|(K_{\Om}(\al-H_{\Om})^{-1})^{n}\|\\
&\quad\quad\geq\int_{\Om\cap B_{\de}(x_{0})}\cdots\int_{\Om\cap B_{\de}(x_{0})}\prod_{m=1}^{n}\bigg[J(x_{m-1}-x_{m})\frac{DM}{\al-(-D+a_{T}(x_{m}))}\bigg]dx_{n}\cdots dx_{1}\\
&\quad\quad\geq\bigg[\inf_{x\in\Om\cap B_{\de}(x_{0})}\int_{\Om\cap B_{\de}(x_{0})}J(x-y)\frac{DM}{\al-(-D+a_{T}(y))}dy\bigg]^{n},
\end{split}
\end{equation*}
where $B_{\de}(x_{0})$ is the open ball in $\R^{N}$ centered at $x_{0}$ with radius $\de$. We then use \eqref{spectral-radii-bound-b} and Gelfand's formula for the spectral radius of a bounded linear operator to obtain
\begin{equation}\label{uniform-estimate-0001-b}
1\geq\inf_{x\in\Om\cap B_{\de}(x_{0})}\int_{\Om\cap B_{\de}(x_{0})}J(x-y)\frac{DM}{\al-(-D+a_{T}(y))}dy=:I(x_{0},\de,\al)
\end{equation}
for all $x_{0}\in\ol{\Om}$, $\de>0$ and $\al\in(-\la_{*},-\la_{*}+1]$.

Since $J$ is continuous and $J(0)>0$, there exists $\de_{*}>0$ and $c_{*}>0$ such that $J\geq c_{*}$ on $B_{\de_{*}}(0)$, the open ball in $\R^{N}$ centered at $0$ with radius $\de_{*}$. Hence, 
\begin{equation*}
\begin{split}
I(x_{0},\de,\al)&\geq\inf_{x\in\Om\cap B_{\de}(x_{0})}\int_{\Om\cap B_{\de}(x_{0})\cap B_{\de_{*}}(x)}J(x-y)\frac{DM}{\al-(-D+a_{T}(y))}dy\\
&\geq c_{*}\inf_{x\in\Om\cap B_{\de}(x_{0})}\int_{\Om\cap B_{\de}(x_{0})\cap B_{\de_{*}}(x)}\frac{DM}{\al-(-D+a_{T}(y))}dy\\
&=c_{*}\int_{\Om\cap B_{\de}(x_{0})}\frac{DM}{\al-(-D+a_{T}(y))}dy
\end{split}
\end{equation*}
provided $2\de\leq\de_{*}$ so that $B_{\de}(x_{0})\subset B_{\de_{*}}(x)$ whenever $x\in\ol{B_{\de}(x_{0})}$. In particular, for any $x_{0}\in\ol{\Om}$ and $\al\in(-\la_{*},-\la_{*}+1]$,
\begin{equation*}
I(x_{0},\de_{*}/2,\al)\geq c_{*}\int_{\Om\cap B_{\de_{*}/2}(x_{0})}\frac{DM}{\al-(-D+a_{T}(y))}dy.
\end{equation*}

Since $\frac{1}{\max_{y\in\ol{\Om}}a_{T}(y)-a_{T}}\notin L^{1}_{loc}(\ol{\Om})$, or equivalently $\frac{1}{-\la_{*}-(-D+a_{T})}\notin L^{1}_{loc}(\ol{\Om})$, there exists $x_{*}\in\ol{\Om}$ such that
\begin{equation*}
\frac{1}{-\la_{*}-(-D+a_{T})}\notin L^{1}(\ol{\Om}\cap B_{\de_{*}/2}(x_{*})),
\end{equation*}
which implies the existence of some $\ep_{*}\in(0,1)$ such that
\begin{equation*}
c_{*}\int_{\Om\cap B_{\de_{*}/2}(x_{*})}\frac{DM}{-\la_{*}+\ep_{*}-(-D+a_{T}(y))}dy\geq2
\end{equation*}
for all $\ep\in(0,\ep_{*}]$. In particular, $I(x_{*},\de_{*}/2,-\la_{*}+\ep_{*})\geq2$, which contradicts to \eqref{uniform-estimate-0001-b}.
\end{proof}

To further investigate the properties of the principal eigenvalue $\la_{1}(-L_{\Om})$ under the condition \eqref{newcond1}, let us introduce the following notions.

\begin{definition}\label{defn-generalized-pe-b}
The {\rm generalized principal eigenvalues} of $-L_{\O}$ are defined as follows:
\begin{equation*}
\begin{split}
\lambda_{p}(-L_{\Om}):&=\sup\left\{\lambda\in\R:\exists\phi\in \mathcal{X}^{++}_{\Om}\,\,\text{s.t.}\,\,(L_{\Om}+\la)[\phi]\leq0\,\,\text{in ${\R\times\ol\O}$}\right\},\\
\lambda_{p}'(-L_{\Om}):&=\inf\left\{\lambda\in\R:\exists\phi\in \mathcal{X}^{++}_{\Om}\,\,\text{s.t.}\,\,(L_{\Om}+\la)[\phi]\geq0\,\,\text{in ${\R\times\ol\O}$}\right\}.
\end{split}
\end{equation*}
A pair $(\la,\phi)\in\R\times\XX_{\Om}^{++}$ is called a {\rm test pair} for $\la_{p}(-L_{\Om})$ (resp. $\la_{p}'(-L_{\Om})$) if $(L_{\Om}+\la)[\phi]\leq0$ in $\R\times\ol\O$ (resp. $(L_{\Om}+\la)[\phi]\geq0$ in $\R\times\ol\O$).
\end{definition}

It is easy to see that $\la_{p}(-L_{\Om})$ and $\la_{p}'(-L_{\Om})$ are always well-defined. We prove the ``Moreover" part of Theorem \ref{thm-pe-introduction}, which is restated in the following theorem.


\begin{theorem}\label{thm-principal-vs-generalized}
Suppose {\bf(H1)}, {\bf(H2)} and \eqref{newcond1}. There holds
\begin{equation*}\label{eigen1}
\la_{p}(-L_{\Om})=\lambda_p'(-L_{\Om})=\la_{1}(-L_{\Om}).
\end{equation*}
\end{theorem}
\begin{proof}
For simplicity, we write $\la_{p}=\la_{p}(-L_{\Om})$, $\la_{p}'=\la_{p}'(-L_{\Om})$ and $\la_{1}=\la_{1}(-L_{\Om})$.  First, we prove 
$$
\la_{1}=\lambda_p.
$$
By Theorem \ref{thm-new-sufficient-cond-b}, there exists $\phi_1\in \XX^{++}_{\Om}$ such that 
\begin{equation}\label{pe-eqn-03-14}
L_{\Om}[\phi_{1}]+\la_{1}\phi_{1}=0\quad\text{in}\quad \R\times\ol{\Om}.
\end{equation}
Since $\inf_{\R\times\ol{\Om}}\phi_1>0$, one has $\la_{1}\leq\la_{p}$. We suppose by contradiction that $\la_{1}<\la_{p}$. From the definition of $\la_{p}$, we can find some $\la\in(\la_{1},\la_{p})$ and $\phi\in\XX_{\Om}^{++}$ such that 
\begin{equation}\label{pe-eqn-aux-03-14}
L_{\Om}[\phi]+\la\phi\leq0\quad\text{in}\quad\R\times\ol{\Om}. 
\end{equation}
Clearly, $w:=\frac{\phi_{1}}{\phi}\in\XX_{\Om}^{++}$.

Rewriting \eqref{pe-eqn-aux-03-14} as
\begin{equation*}
-\phi_{t}+a(t,x)\phi\leq-\la\phi-D\left[\int_{\Om}J(x-y)\phi(t,y)dy-\phi(t,x)\right],
\end{equation*}
we deduce
\begin{equation*}
\begin{split}
L_{\Om}[\phi_{1}]&=-w_{t}\phi+D\left[\int_{\Om}J(x-y)\phi(t,y)w(t,y)dy-w(t,x)\phi(t,x)\right]+[-\phi_{t}+a(t,x)\phi(x)]w\\
&\leq-w_{t}\phi+D\left[\int_{\Om}J(x-y)\phi(t,y)w(t,y)dy-w(t,x)\phi(t,x)\right] \\
&\quad+\left[-\la\phi-D\left(\int_{\Om}J(x-y)\phi(t,y)dy-\phi(t,x)\right)\right] w\\
&=-w_{t}\phi+D\int_{\Om}J(x-y)\phi(t,y)[w(t,y)-w(t,x)]dy-\la\phi_{1}.
\end{split}
\end{equation*}
Using \eqref{pe-eqn-03-14}, we find
\begin{equation}\label{an-inequality-000001}
-(\la_{1}-\la)\phi_{1}\leq-w_{t}\phi+D\int_{\Om}J(x-y)\phi(t,y)[w(t,y)-w(t,x)]dy.
\end{equation}

As $w\in\XX_{\Om}^{++}$, there exists $(t_{0},x_{0})\in\R\times\ol{\Om}$ such that
\begin{equation*}\label{23.11.1}
w(t_{0},x_{0})=\max_{(t,x)\in\R\times\ol{\Om}}w(t,x). 
\end{equation*}
Then, $w_{t}(t_{0},x_{0})=0$. Hence, setting $(t,x)=(t_{0},x_{0})$ in \eqref{an-inequality-000001} yields $-(\la_{1}-\la)\phi_{1}(t_{0},x_{0})\leq0$, which leads to $\la_{1}\geq\la$. This contradiction confirms $\la_{1}=\lambda_p$.

Next, we prove
$$
\lambda_1=\lambda_p'.
$$
Obviously $\lambda_1\geq\lambda_p'$. Assume that $\lambda_1>\lambda_p'$. One can find some $\tilde{\lambda}\in (\lambda_p',\lambda_1)$ and $\tilde{\phi}\in\XX_{\Om}^{++}$ such that $L[\tilde{\phi}]+\tilde{\la}\tilde{\phi}\geq0$. Set $\tilde{w}:=\frac{\phi_1}{\tilde{\phi}}$. The same arguments as above apply and we derive
\begin{equation}\label{22.11.1}
0>-(\la_{1}-\tilde{\la})\phi_{1}\geq-\tilde{w}_{t}\phi+D\int_{\Om}J(x-y)\tilde{\phi}(t,y)[\tilde{w}(t,y)-\tilde{w}(t,x)]dy.
\end{equation}
One can find some $(t_1,x_1)\in \R\times\ol{\Om}$ such that
\begin{equation*}
\tilde{w}(t_{1},x_{1})=\min_{(t,x)\in\R\times\ol{\Om}}\tilde{w}(t,x). 
\end{equation*}
Substituting $(t_1,x_1)$ into the right-hand side of \eqref{22.11.1}, we derive the contradiction.
\end{proof}

We remark that the parabolic-type operator $-L_{\Om}$ is not self-adjoint, and thus, we are lack of the usual $L^{2}(\Om)$ variational formula for the principal eigenvalue $\la_{1}(-L_{\Om})$. The $\sup$-$\inf$ characterizations of $\la_{1}(-L_{\Om})$ given in Theorem \ref{thm-principal-vs-generalized} remedy the situation and  play crucial roles in the sequel.


\section{Bifurcation and global dynamics}\label{sec-bifurcation-global-dynamics}

In this section, we study the long-time behaviors of the solutions of \eqref{main-eqn-Dirichlet-simple1}, namely, 
\begin{equation*}
u_{t}(t,x)=D\left[\int_{\Om}J(x-y)u(t,y)dy-u(t,x)\right]+f(t,x,u(t,x)),\quad t>0,\quad x\in\ol{\Om},
\end{equation*}
and prove Theorem \ref{thm-persistence-criterion-b}.

We define the following spaces:
\begin{equation*}
\begin{split}
X_{\Om}&=C(\ol{\Om}),\\
X_{\Om}^{+}&=\big\{v\in X_{\Om}\big|v(x)\geq0,\,\,x\in\ol{\Om}\big\},\quad\text{and}\\
X_{\Om}^{++}&=\big\{v\in X_{\Om}\big|v(x)>0,\,\,x\in\ol{\Om}\big\}.
\end{split}
\end{equation*}
Denote by $\|\cdot\|_{\infty}$ the max norm on $X_{\Om}$. For $u_{0}\in X_{\Om}$, we denote by $u(t,\cdot;u_{0})\in X_{\Om}$ for all $t>0$ the unique solution of \eqref{main-eqn-Dirichlet-simple1} with initial data $u(0,\cdot;u_{0})=u_{0}$. By the comparison principle (see e.g. \cite{HSV08,RS12}), if $u_{0}\in X_{\Om}^{+}$, then $u(t,\cdot;u_{0})\in X_{\Om}^{+}$ for all $t>0$. Moreover, if $u_{0}\in X_{\Om}^{+}\bs\{0\}$, then $u(t,\cdot;u_{0})\in X_{\Om}^{++}$ for all $t>0$. 

Let us consider the linearization of \eqref{main-eqn-Dirichlet-simple1} at $u\equiv0$, namely,
\begin{equation}\label{main-eqn-linear1-b}
u_{t}(t,x)=D\left[\int_{\O}J(x-y)u(t,y)dy-u(t,x)\right]+a(t,x)u(t,x),\quad t>0,\quad x\in\ol{\O}.
\end{equation}
Denote by $\{\Phi(t;s)\}_{t\geq s\geq0}$ the evolution family on $X_{\Om}$ generated by \eqref{main-eqn-linear1-b}, that is, if $u(t,x;s,u_{0})$ is the unique solution of \eqref{main-eqn-linear1-b} with initial data $u(s,\cdot;s,u_{0})=u_{0}\in X_{\Om}$, then $u(t,\cdot;s,u_{0})=\Phi(t;s)u_{0}\in X_{\Om}$ for all $t\geq s$. By comparison principle, if $u_{0}\in X_{\Om}^{+}$, so does $\Phi(t;s)u_{0}$ for all $t>s$. Moreover, if $u_{0}\in X_{\Om}^{+}\bs\{0\}$, then $\Phi(t;s)u_{0}\in X^{++}_{\Om}$ for all $t>s$. Also, by time-periodicity, one has $\Phi(t+T,s+T)=\Phi(t,s)$ for all $t\geq s\geq0$. The operator norm of $\Phi(t,s)$ is denoted by $\|\Phi(t,s)\|$.

To prove Theorem \ref{thm-persistence-criterion-b}, we first prove the following two comparison principles.

\begin{proposition}\label{Prop1}
Let $u\in\XX_{\Om}^{++}$ be a sub-solution of \eqref{main-eqn-Dirichlet-simple} and $v\in\XX_{\Om}^{++}$ be a super-solution of \eqref{main-eqn-Dirichlet-simple}. Then, $u\leq v$ in $\R\times\ol{\Om}$.
\end{proposition}
\begin{proof}
Let
$$
\al_{*}:=\sup\left\{\al>0:\al u\leq v\,\,\text{in}\,\,\R\times\ol{\Om}\right\}.
$$
By the assumptions on $u$ and $v$, the number $\al_{*}$ is well-defined and positive. If $\al_{*}\leq1$, then we are done. So, we assume $\al_{*}>1$.

Set $w:=v-\al_{*}u$. Then, $w\geq0$ and there exists $(t_{0},x_{0})\in\R\times\ol{\Om}$ such that $w(t_{0},x_{0})=0$. Obviously, $w$ satisfies
\begin{equation*}
\begin{split}
w_{t}(t,x)&\geq D\left[\int_{\Om}J(x-y)w(t,y)dy-w(t,x)\right]+f(t,x,v(t,x))-\al_{*}f(t,x,u(t,x))\\
&>D\left[\int_{\Om}J(x-y)w(t,y)dy-w(t,x)\right]+f(t,x,v(t,x))-f(t,x,\al_{*}u(t,x)),\quad (t,x)\in\R\times\ol{\Om},
\end{split}
\end{equation*}
where we used {\bf(H2)}-(3) and $\al_{*}>1$ in the second inequality. Considering the above inequality at $(t_{0},x_{0})$, we find the contradiction immediately.
\end{proof}

\begin{proposition}\label{Prop1-2}
Let $u$ be a nonnegative and bounded solution of \eqref{main-eqn-Dirichlet-simple} and be equi-continuous in space, namely, for any $\ep>0$ there exists $\de>0$ such that 
$$
x,y\in\ol{\Om},\,\,|x-y|\leq\de\quad\text{implies}\quad|u(t,x)-u(t,y)|\leq\de,\,\,\forall t\in\R.
$$
If $v\in\XX_{\Om}^{++}$ is a super-solution of \eqref{main-eqn-Dirichlet-simple}, then $u\leq v$ in $\R\times\ol{\Om}$.
\end{proposition}
\begin{proof}
By the assumptions on $u$ and $v$, it is obviously seen that
$$
k_{*}:=\inf\{k>0|kv\geq u\text{ in}\,\,\R\times\ol{\Om}\} 
$$ 
is well-defined and positive. 

Let us assume by contradiction that $k_{*}>1$. We see that $w:=u-k^{*}v\leq0$ on $\R\times\ol{\Om}$, and by the definition of $k^*$, there exists a sequence $\{(t_n,x_n)\}\subset\R\times\ol{\Om}$ such that
$$
w(t_{n},x_{n})=u(t_n,x_n)-k^{*}v(t_n,x_n)\to0\quad\text{as}\quad n\to\infty.
$$
Let $\tau_n\in[0,T]$ be such that $t_n-\tau_n\in T\Z$. Up to a subsequence, we may assume, without loss of generality, that $\tau_n\to\tau_{*}\in[0,T]$, $x_n\to x_{*}\in\ol{\Om}$. Considering the functions 
$$
u_n(t,x)=u(t+t_n,x),\quad v_n(t,x)=v(t+t_n,x)\quad\text{and}\quad w_n(t,x)=w(t+t_n,x), 
$$
we see that $w_n$ satisfies
\begin{equation}\label{inequality-w-03-14}
\begin{split}
&\partial_t w_n(t,x)-D\left[\int_{\Om}J(x-y)w_n(t,y)dy-w_n(t,x)\right]\\
&\quad\quad\leq f(t+t_n,x,u_n(t,x))-k_*f(t+t_n,x,v_n(t,x)),\quad (t,x)\in\R\times\ol{\Om}.
\end{split}
\end{equation}
Since $v$ is $T$-periodic in $t$, $v_n(t,x)$ converges, uniformly in $(t,x)\in\R\times\ol{\Om}$, as $n\to\infty$, to $\ol v(t,x)=v(t+\tau_{*},x)$. Moreover, $\partial_{t}v_n(t,x)$ converges, uniformly in $(t,x)\in\R\times\ol{\Om}$, as $n\to\infty$, to $\ol v_{t}(t,x)$.

By the assumptions on $u$ and {\bf(H2)}-(1), $\{u_{n}\}_{n}$, $\{\partial_{t}u_{n}\}_{n}$ are $\{\partial_{tt}^{2}u_{n}\}_{n}$  are uniformly bounded and equi-continuous in space and time. Arzel\`{a}-Ascoli theorem then implies the existence of some continuous function $\ol{u}:\R\times\ol{\Om}\to\R$, continuously differentiable in $t$, such that $u_n(t,x)\to \ol u(t,x)$ and $\partial_{t}u_n(t,x)\to \ol u_{t}(t,x)$ locally uniformly in $(t,x)\in\R\times\ol{\Om}$ as $n\to\infty$ along some subsequence. In particular, $\ol{u}$ satisfies \eqref{main-eqn-Dirichlet-simple}. Therefore, setting $\ol w:=\ol u-k_*\ol v$ and letting $n\to\infty$ in \eqref{inequality-w-03-14}, we find
\begin{equation}\label{29.11.1}
\begin{split}
&\partial_t \ol w(t,x)-D\left[\int_{\Om}J(x-y)\ol w(t,y)dy-\ol w(t,x)\right]\\
&\quad\quad\leq f(t+\tau_{*},x,\ol u(t,x))-k_*f(t+\tau_{*},x,\ol v(t,x))\\
&\quad\quad< f(t+\tau_{*},x,\ol u(t,x))-f(t+\tau_{*},x,k_*\ol v(t,x))\\
&\quad\quad= \tilde{a}(t,x)\ol w(t,x),\quad (t,x)\in\R\times\ol{\Om},
\end{split}
\end{equation}
where we used {\bf(H2)}-(3) and 
$$
\tilde{a}(t,x)=\begin{cases}
0,&\quad \ol{u}(t,x)=k_{*}\ol{v}(t,x),\\
\frac{f(t+\tau_{*},x,\ol u(t,x))-f(t+\tau_{*},x,k^*\ol v(t,x))}{\ol u(t,x)-k_*\ol v(t,x)},&\quad\ol{u}(t,x)=k_{*}\ol{v}(t,x).
\end{cases}
$$

Obviously, $\ol w(0,x_{*})=0$, the maximum principle (see Theorem 3.2 \cite{HMMV03}), asserts that there exists $t^*\in[0,T]$ such that $\ol{w}(t,x)=0$ for $(t,x)\in[-t^*,0]\times\ol{\Om}$. This contradicts to (\ref{29.11.1}). Hence, $k_{*}\leq1$ and $v\geq u$ in $\R\times\ol\O$ follows.
\end{proof}

We are now in the position to prove Theorem \ref{thm-persistence-criterion-b}.

\begin{proof}[{\bf Proof of Theorem \ref{thm-persistence-criterion-b}}]
Let $\la_{1}=\la_{1}(-L_{\Om})$ for simplicity. 

If $\la_{1}<0$, the results in \cite[Theorem E]{RS12} using a contraction argument confirms the existence and uniqueness of a solution $u^{*}\in\XX_{\Om}^{++}$ of \eqref{main-eqn-Dirichlet-simple} as well as the statement in (i). For the sake of completeness, we outline the arguments.
 
On one hand, it is easy to see from {\bf(H2)}-(4) that for any $M\gg1$, $u(t,x)\equiv M$ is a super-solution of \eqref{main-eqn-Dirichlet-simple}, and then, the time-periodicity implies that $\{u(nT,\cdot;M)\}_{n}$ is a non-increasing sequence. Therefore,
$$
u^{+}(x):=\lim_{n\to\infty}u(nT,x;M),\quad x\in\ol{\Om}
$$
is well-defined and upper semi-continuous.  On the other hand, for any $0<\ep\ll1$, it can be shown using the assumption $\la_{1}<0$ that $\ep\phi_{1}$ is a sub-solution of \eqref{main-eqn-Dirichlet-simple}, where $\phi_{1}$ is a fixed principal eigenfunction of $-L_{\Om}$, and then, $\{u(nT,\cdot;\ep\phi_{1}(0,\cdot))\}_{n}$ is a non-decreasing sequence. Therefore,
$$
u^{-}(x):=\lim_{n\to\infty}u(nT,x;\ep\phi_{1}(0,\cdot)),\quad x\in\ol{\Om}
$$
is well-defined and lower semi-continuous. Clearly, $u^{-}\leq u^{+}$.

Arguments using part metric then ensure $u^{+}=u^{-}$. To be more specific, we define
$$
\rho_{n}:=\inf\left\{\ln\al:\frac{1}{\al}u(nT,x;M)\leq u(nT,x;\ep\phi_{1}(0,\cdot))\leq\al u(nT,x;M)\right\}.
$$
It can be shown that the sequence $\{\rho_{n}\}_{n}$ is decreasing, and thus,
$$
\rho_{*}:=\lim_{n\to\infty}\rho_{n}
$$
is well-defined. It can be further shown that $\rho_{*}=0$, which implies $u^{+}=u^{-}$. 

Hence, $v^{*}:=u^{+}$ is continuous and $\inf_{x\in\ol{\Om}}v^{*}>0$. Clearly, $u(T,\cdot;v^{*})=v^{*}$. Then, $v^{*}$ can be easily extended to a solution $u^{*}\in\XX_{\Om}^{++}$ of \eqref{main-eqn-Dirichlet-simple} such that $u^{*}(t,\cdot)=u(t,\cdot;v^{*})$ for $t\in[0,T]$.

By the above contraction argument, the uniqueness of solutions of \eqref{main-eqn-Dirichlet-simple} in the space $\XX_{\Om}^{++}$ follows. The uniqueness also follows directly from Proposition \ref{Prop1}. The global stability of $u^{*}$ follows again from the contraction argument.

Next, we show that if $\la_{1}\geq0$, then the equation \eqref{main-eqn-Dirichlet-simple} admits no solution in $\XX_{\Om}^{++}$. For contradiction, suppose that $v^{*}\in\XX_{\Om}^{++}$ is a solution of \eqref{main-eqn-Dirichlet-simple}. Let $\phi_1$ be the principal eigenfunction associated to $\la_{1}$ with the normalization $\phi_1<v^{*}$ in $\R\times\ol{\Om}$. We see from {\bf(H2)}-(3) that
\begin{equation*}
\begin{split}
0&\leq\la_{1}\phi_{1}(t,x)\\
&=\partial_t\phi_1(t,x)-D\left[\int_{\Om}J(x-y)\phi_1(t,y)dy-\phi_1(t,x)\right]-a(t,x)\phi_1(t,x)\\
&\leq \partial_t\phi_1(t,x)-D\left[\int_{\Om}J(x-y)\phi_1(t,y)dy-\phi_1(t,x)\right]-f(t,x,\phi_1(t,x)),\quad(t,x)\in\R\times\ol{\Om},
\end{split}
\end{equation*}
that is, $\phi_{1}$ is a super-solution of \eqref{main-eqn-Dirichlet-simple}. By Proposition \ref{Prop1}, there holds $v^{*}\leq \phi_1$ in $\R\times\ol{\Om}$, which contradicts to the normalization.

Finally, we prove $\rm(ii)$ and $\rm(iii)$.

(ii) Suppose $\la_{1}>0$. Since $f(t,x,u(t,x;u_{0}))\leq a(t,x)u(t,x;u_{0})$, one verifies
\begin{equation*}
u_{t}(t,x;u_{0})\leq D\left[\int_{\Om}J(x-y)u(t,y;u_{0})dy-u(t,x;u_{0})\right]+a(t,x)u(t,x;u_{0}).
\end{equation*}
Comparison principle  yields $u(t,\cdot;u_{0})\leq \Phi(t,0)u_{0}$.

We claim $\|\Phi(t,0)u_{0}\|_{\infty}\to0$ as $t\to\infty$. Write $t=[t]+r_{t}$, where $[t]$ is the largest number of the form $nT$ not great than $t$ and $r_{t}\in[0,T)$. By time-periodicity, one finds
\begin{equation*}
\Phi(t,0)u_{0}=\Phi(t,[t])\Phi([t],[t]-T)\cdots\Phi(T,0)u_{0}=\Phi(t_{r},0)\Phi(T,0)^{\frac{[t]}{T}}u_{0}.
\end{equation*}
Obviously, there is $C=C(T)>0$ such that $\|\Phi(t_{r},0)\|\leq C$. It is well-known that 
\begin{equation*}
r(\Phi(T,0))=\lim_{n\to\infty}\|\Phi(T,0)^{n}\|^{\frac{1}{n}}\quad(\text{Gelfand's formula}).
\end{equation*}
Moreover, by \cite[Proposition 3.10]{RS12}, one has 
$$
-\la_{1}=\frac{\ln r(\Phi(T,0))}{T}, 
$$
and therefore, 
$$
e^{-\la_{1}T}=\lim_{n\to\infty}\|\Phi(T,0)^{n}\|^{\frac{1}{n}}. 
$$
In particular, one finds $\|\Phi(T,0)^{n}\|\leq e^{-\frac{\la_{1}}{2}Tn}$ for $n\gg1$. Hence, 
\begin{equation*}
\|\Phi(t,0)u_{0}\|_{\infty}\leq C\|u_{0}\|_{\infty}\|\Phi(T,0)^{\frac{[t]}{T}}\|\leq C\|u_{0}\|_{\infty}e^{-\frac{\la_{1}}{2}[t]}\to0\quad\text{as}\quad t\to\infty.
\end{equation*}
This proves the claim and confirms the statement.

(iii) Let $u(t,x;u_{0})$ be the solution as in the statement. Suppose $\|u(t,\cdot;u_{0})\|_{\infty}\nrightarrow0$ as $t\to\infty$. Then, we can find  some  $\de_{0}>0$ and a sequence $\{(t_{n},x_{n})\}\subset[t_{0},\infty)\times\ol{\Om}$ with $t_{n}\to\infty$ as $n\to\infty$ such that
\begin{equation}\label{positivity-cond-001}
u(t_{n},x_{n};u_{0})\geq\de_{0},\quad\forall n.
\end{equation}
Set $u_{n}(t,x)=u(t+t_{n},x;u_{0})$. Arguments as in the proof of Proposition \ref{Prop1} ensure the existence of some continuous function $u:\R\times\ol{\Om}\to\R$, continuously differentiable in $t$, such that $u_n(t,x)\to u(t,x)$ and $\partial_{t}u_n(t,x)\to u_{t}(t,x)$ locally uniformly in $(t,x)\in\R\times\ol{\Om}$ as $n\to\infty$ along some subsequence (still denoted by $n\to\infty$). Clearly, $u(t,x)$ is a nonnegative and bounded solution of \eqref{main-eqn-Dirichlet-simple}.

We claim that $u\equiv0$. If not, then we can normalize $\phi_{1}$ so that there is some $(t_{*},x_{*})\in\R\times\ol{\Om}$ such that $u(t_{*},x_{*})>\phi_{1}(t_{*},x_{*})$. As in  (2), we can show that $\phi_{1}$ is super-solution of \eqref{main-eqn-Dirichlet-simple}. Applying Proposition \ref{Prop1-2}, we then conclude that $u\leq\phi_{1}$, which leads to a contradiction.

But, the estimate \eqref{positivity-cond-001} says that passing $n\to\infty$ along some subsequence, we find some $x_{*}\in\ol{\Om}$ such that 
$$
u(0,x_{*})\leftarrow u(0,x_{n})=u(t_{n},x_{n};u_{0})\geq\de_{0},
$$
which contradicts to $u\equiv0$. Hence, $\lim_{t\to\infty}\|u(t,\cdot;u_{0})\|_{\infty}=0$.
\end{proof}


\section{Effects of the dispersal rate}\label{sec-effect-dispersal-rate}

In this section, we study the effects of the dispersal rate $D$ on $\la_{1}(D):=\la_{1}(-L_{\Om})$ and the positive $T$-periodic solution associated to the equations \eqref{main-eqn-Dirichlet-simple1} and \eqref{main-eqn-Dirichlet-simple}. In particular, we prove Theorem \ref{thm-effect-dispersal-rate-intro}.

We first prove Theorem \ref{thm-effect-dispersal-rate-intro}(1) concerning the effects of $D$ on $\la_{1}(D)$. 

\begin{proof}[{\bf Proof of Theorem \ref{thm-effect-dispersal-rate-intro}(1)}]
As $\la_{1}(D)$ is an isolated eigenvalue, the continuous differentiability follows from the classical perturbation theory (see e.g. \cite{Ka95}).

We first prove that
\begin{equation}\label{limit-D-0}
\la_{1}(D)\to-\max_{x\in\ol{\Om}}a_{T}(x)\quad\text{as}\quad D\to0^{+}.
\end{equation}
We claim that, for each $0<\ep\ll1$, there exists $D_{\ep}\in(0,1)$ such that
\begin{equation}\label{bound-on-pev-small-dispersal}
-\max_{x\in\ol{\Om}}a_{T}(x)-\ep\leq\la_{1}(D)\leq-\min_{x\in\ol{\Om}}a_{T}(x)+\ep,\quad\forall D\in(0,D_{\ep}).
\end{equation}
It is easy to check that 
$$
\phi(t,x):=e^{\int_{0}^{t}\left[a(s,x)-a_{T}(x)\right]ds},\quad (t,x)\in\R\times\ol{\Om}.
$$
is a positive $T$-periodic solution of $\phi_{t}=a(t,x)\phi-a_{T}(x)\phi$. In particular, $\phi\in\XX^{++}_{\Om}$. 

For any $0<\ep\ll1$, we set 
$$
\la_{\ep}^{\max}=-\max_{x\in\ol{\Om}}a_{T}(x)-\ep\quad\text{and}\quad\la_{\ep}^{\min}=-\min_{x\in\ol{\Om}}a_{T}(x)+\ep.
$$
Using $\min_{(t,x)\in[0,T]\times\ol{\Om}}\phi(t,x)>0$, it is straightforward to check that for each $0<\ep\ll1$, there exists $0<D_{\ep}\ll1$ such that for each $D\in(0,D_{\ep})$, there hold
\begin{equation}\label{test-pairs}
(L_{\Om}+\la_{\ep}^{\max})[\phi]\leq0\quad\text{and}\quad (L_{\Om}+\la_{\ep}^{\min})[\phi]\geq0.
\end{equation}
It then follows from \eqref{test-pairs}, the definitions of $\la_{p}(-L_{\Om})$ and $\la_{p}'(-L_{\Om})$, and Theorem \ref{thm-principal-vs-generalized} that for each $0<\ep\ll1$,
$$
\la_{\ep}^{\max}\leq \la_{1}(D)\leq\la_{\ep}^{\min},\quad \forall D\in(0,D_{\ep}).
$$
This is exactly \eqref{bound-on-pev-small-dispersal}.

Now, by Theorem \ref{thm-new-sufficient-cond-b} and \eqref{bound-on-pev-small-dispersal}, for each $0<\ep\ll1$ there exists $D_{\ep}\in(0,1)$ such that
$$
-\max_{x\in\ol{\Om}}a_{T}(x)-\ep\leq\la_{1}(D)\leq\min_{x\in\ol{\Om}}\left[D-a_{T}(x)\right],\quad\forall D\in(0,D_{\ep}).
$$
Setting $D\to0^{+}$, we find
$$
-\max_{x\in\ol{\Om}}a_{T}(x)-\ep\leq\liminf_{D\to0^{+}}\la_{1}(D)\leq\limsup_{D\to0^{+}}\la_{1}(D)\leq-\max_{x\in\ol{\Om}}a_{T}(x),\quad\forall0<\ep\ll1,
$$
which leads to \eqref{limit-D-0}.

To show
\begin{equation}\label{limit-D-infty}
\la_{1}(D)\to\infty\quad\text{as}\quad D\to\infty,
\end{equation}
we consider the following operator
$$
L_{\Om}^{0}[\psi]:=\int_{\Om}J(\cdot-y)\psi(y)dy-\psi(x),\quad \psi\in C(\Om),
$$
where $C(\Om)$ is the space of continuous functions on $\Om$. It is known from \cite[Theorem 2.1 and Proposition 3.4]{ShXi15-1} that the principal eigenvalue of $-L_{\Om}^{0}$ exists and is positive. Let $\la^{0}>0$ be the principal eigenvalue of $-L_{\Om}^{0}$, and $\psi^{0}\in C^{++}(\Om)$ be an associated eigenfunction, where $C^{++}(\Om)=\{\psi\in C(\Om):\inf_{\Om}\psi>0\}$. 

Let 
$$
\la_{D}=D\la^{0}-\max_{(t,x)\in[0,T]\times\ol{\Om}}a(t,x). 
$$
We see that
\begin{equation*}
(L_{\Om}+\la_{D})[\psi^{0}]=DL_{\Om}^{0}[\psi^{0}]+a(t,x)\psi^{0}+\la_{D}\psi^{0}=[-D\la^{0}+a(t,x)+\la_{D}]\psi^{0}\leq0.
\end{equation*}
That is, $(\la_{D},\psi^{0})$ is a test pair for $\la_{p}(-L_{\Om})$. It then follows that $\la_{1}(D)=\la_{p}(-L_{\Om})\geq\la_{D}$. Setting $D\to\infty$, we arrive at \eqref{limit-D-infty}.
\end{proof}

Next, we prove the monotonicity of the function $D\mapsto\la_{1}(D)$ for a special class of $a(t,x)$ as in Theorem \ref{thm-effect-dispersal-rate-intro}(2). 

\begin{proof}[{\bf Proof of Theorem \ref{thm-effect-dispersal-rate-intro}(2)}]
We write
$$
L_{\Om}=L_{\Om}^{\CS}+L^{\CT}_{\Om},
$$
where the superscripts $\CS$ and $\CT$ stand for space and time, respectively, and
\begin{equation*}
\begin{split}
L_{\Om}^{\CS}[v](x)&=D\left[\int_{\Om}J(x-y)v(y)dy-v(x)\right]+\beta(x)v(x),\\
L_{\Om}^{\CT}[v](t)&=-v_{t}(t)+\al(t)v(t).
\end{split}
\end{equation*}

Let $(\la_{1}^{\CS}(D),\phi_{D}^{\CS})$ be the principal eigen-pair of $-L_{\Om}^{\CS}$. It is known from \cite[Theorem 2.2 (1)]{ShXi15-1} that $D\mapsto\la_{1}^{\CS}(D)$ is non-decreasing, and increasing if $\beta$ is not a constant function. Let us show that the function $D\mapsto\la_{1}^{\CS}(D)$ is increasing even if $\beta$ is a constant function. From the classical perturbation theory (see e.g. \cite{Ka95}), we know that $D\mapsto(\la_{1}^{\CS}(D),\phi_{D}^{\CS})$ is continuously differentiable, and therefore, we can differentiate the equation $L_{\Om}^{\CS}[\phi_{D}^{\CS}]=\la_{1}^{\CS}(D)\phi_{D}^{\CS}$ with respect to $D$ to find
$$
\int_{\Om}J(x-y)\phi_{D}^{\CS}(y)dy-\phi_{D}^{\CS}(x)+L_{\Om}^{\CS}[\partial_{D}\phi_{D}^{\CS}](x)+\partial_{D}\la_{1}^{\CS}(D)\phi_{D}^{\CS}(x)+\la_{1}^{\CS}(D)\partial_{D}\phi_{D}^{\CS}(x)=0. 
$$
Multiplying the above equation by $\phi_{D}^{\CS}$ and integrating the resulting equation over $\Om$, we find 
\begin{equation}\label{equation-for-partial-D-pv}
\int_{\Om\times\Om}J(x-y)\phi_{D}^{\CS}(x)\phi_{D}^{\CS}(y)dxdy-\int_{\Om}\phi_{D}^{\CS}(x)^{2}dx+\partial_{D}\la_{1}^{\CS}(D)\int_{\Om}\phi_{D}^{\CS}(x)^{2}dx=0,
\end{equation}
where we used the fact that
\begin{equation*}
\begin{split}
&\int_{\Om}L_{\Om}^{\CS}[\partial_{D}\phi_{D}^{\CS}](x)\phi_{D}^{\CS}(x)dx+\int_{\Om}\la_{1}^{\CS}(D)\partial_{D}\phi_{D}^{\CS}(x)\phi_{D}^{\CS}(x)dx\\
&\quad\quad=\int_{\Om}\partial_{D}\phi_{D}^{\CS}(x)\left[L_{\Om}^{\CS}[\phi_{D}^{\CS}](x)+\la_{1}^{\CS}(D)\phi_{D}^{\CS}(x)\right]dx=0.
\end{split}
\end{equation*}
It then follows from \eqref{equation-for-partial-D-pv} that 
$$
\partial_{D}\la_{1}^{\CS}(D)=-\frac{\int_{\Om\times\Om}J(x-y)\phi_{D}^{\CS}(x)\phi_{D}^{\CS}(y)dxdy-\int_{\Om}\phi_{D}^{\CS}(x)^{2}dx}{\int_{\Om}\phi_{D}^{\CS}(x)^{2}dx}.
$$
As 
$$
\int_{\Om\times\Om}J(x-y)\phi_{D}^{\CS}(x)\phi_{D}^{\CS}(y)dxdy<\int_{\Om}\phi_{D}^{\CS}(x)^{2}dx,
$$
we conclude that $\partial_{D}\la_{1}^{\CS}(D)>0$, that is, $D\mapsto\la_{1}^{\CS}(D)$ is increasing.

Now, let us define
$$
\phi_{D}^{\CT}(t)=e^{\int_{0}^{t}[\al(s)-\al_{T}]ds},\quad t\in\R,
$$
where $\al_{T}=\frac{1}{T}\int_{0}^{T}\al(t)dt$. Clearly, $\phi_{D}^{\CT}$ is continuously differentiable, positive and $T$-periodic, and satisfies
$$
L_{\Om}^{\CT}[\phi_{D}^{\CT}]+\al_{T}\phi_{D}^{\CT}=0.
$$

It is then clear that $\la_{1}(D)=\la_{1}^{\CS}(D)+\al_{T}$ is the principal eigenvalue of $-L_{\Om}$ with the principal eigenfunction $\phi_{D}(t,x)=\phi_{D}^{\CS}(x)\phi_{D}^{\CT}(t)$. The result of the lemma then follows from the property of $D\mapsto\la_{1}^{\CS}(D)$.
\end{proof}



Finally, we study the effects of the dispersal rate $D$ on the positive $T$-periodic solutions of \eqref{main-eqn-Dirichlet-simple} by proving Theorem \ref{thm-effect-dispersal-rate-intro}(3)(4). 

\begin{proof}[{\bf Proof of Theorem \ref{thm-effect-dispersal-rate-intro}(3)(4)}]
(3) By Theorem \ref{thm-effect-dispersal-rate-intro}(1), $\la_{1}(D)<0$ for all $0<D\ll1$ and $\la_{1}(D)>0$ for all $D\gg1$. The result then follows from Theorem \ref{thm-persistence-criterion-b}.

(4) We refer the reader to the proof of Lemma \ref{lem-sol-limit-eqn} and Theorem \ref{thm-scaling-limit-positive-sol}, where similar problems are treated.
\end{proof}



\section{Effects of the dispersal range and scaling limits}\label{sec-scaling-limits}

In this section, we study the effects of dispersal range characterized by $\si$ on the principal eigenvalue and the positive $T$-periodic solution associated to \eqref{main-eqn} and \eqref{main-eqn-R}. In particular, we prove Theorem \ref{thm-scaling-limit-eigenvalue} and Theorem \ref{thm-scaling-limits-introduction}.

Let us start with the consideration of the following operator
\begin{equation}\label{main-eqn-linear-variation}
L_{\Om,m,\si}[v](t,x):=\M_{\O,m,\sigma}[v](t,x)+a(t,x)v(t,x),\quad (t,x)\in\R\times\ol{\Om}
\end{equation}
associated to the linearization of \eqref{main-eqn} at $u\equiv0$, where
$$
\M_{\O,m,\sigma}[v](t,x):=-v_{t}(t,x)+\frac{D}{\si^{m}}\left[\int_{\Om}J_{\si}(x-y)v(t,y)dy-v(t,x)\right],\quad (t,x)\in\R\times\ol{\Om}.
$$

Note that for each fixed $\si>0$ and $m\geq0$, introducing the new dispersal rate $\tilde{D}:=\frac{D}{\si^{m}}$ and the new dispersal kernel $\tilde{J}:=J_{\si}$, we are completely in the situation of the operator \eqref{main-eqn-linear}. Therefore, studies and results in the previous sections apply here. In particular, we have the following proposition collecting some basic facts about $\la_{1}(-L_{\Om,m,\si})$, $\la_{p}(-L_{\Om,m,\si})$ and $\la_{p}'(-L_{\Om,m,\si})$.

\begin{proposition}\label{prop-principal-e-variation-1}
Suppose {\bf(H1)} and {\bf(H2)}. Let $m\geq0$ and $\si>0$.
\begin{enumerate}
\item The principal spectrum point $\la_{1}(-L_{\Om,m,\si})$ is the principal eigenvalue if and only if  
$$
\quad\la_{1}(-L_{\Om,m,\si})<\frac{D}{\si^{m}}-\max_{x\in\ol{\Om}}a_{T}(x).
$$ 
Moreover, when $\la_{1}(-L_{\Om,m,\si})$ is the principal eigenvalue of $-L_{\Om,m,\si}$, it is geometrically simple and has an associated eigenfunction in $\XX_{\Om}^{++}$.

\item If \eqref{newcond1} holds, then $\la_{1}(-L_{\Om,m,\si})$ is the principal eigenvalue of $-L_{\Om,m,\si}$ and there holds
$$
\la_{p}(-L_{\Om,m,\si})=\lambda_p'(-L_{\Om,m,\si})=\la_{1}(-L_{\Om,m,\si}).
$$

\item  $\la_{1}(-\M_{\O,m,\sigma}-a)$ is Lipschitz continuous with respect to $a(t,x)$. More precisely,
$$
|\la_{1}(-\M_{\O,m,\sigma}-a)-\la_{1}(-\M_{\O,m,\sigma}-b)|\leq \sup_{t\in[0,T]}\|a(t,\cdot)-b(t,\cdot)\|_{\infty}.
$$

\item If $\O_1\subset\O_2$, then $\lambda_p(-L_{\O_1,m,\sigma})\geq \lambda_p(-L_{\O_2,m,\sigma})$. If, in addition, \eqref{newcond1} holds in $\O_1$ and $\O_2$, then
$$
|\lambda_p(-L_{\O_2,m,\sigma})-\lambda_p(-L_{\O_1,m,\sigma})|\leq C_0|\O_2 \setminus \O_1|,
$$
where $C_0>0$ depending on $m$, $\si$, $J$ and the associated principal eigenfunction of $\lambda_p(-L_{\O_2,m,\sigma})$.

\item If \eqref{newcond1} holds, then the function $\si\mapsto\la_{1}(-L_{\Om,m,\si})$ is continuous.

\end{enumerate}
\end{proposition}
\begin{proof}
(1) It is a direct consequence of Theorem \ref{prop-principal-e}.  

(2) It follows  from Theorem \ref{thm-new-sufficient-cond-b} and Theorem \ref{thm-principal-vs-generalized}.  

(3) It can be found in Proposition 3.3 in \cite{ShXi15-2}. 

(4) Let $(\la,\psi)$ be a test pair for $\la_{p}(-L_{\Omega_2,m,\si})$, that is, $(\la,\psi)\in\R\times\XX_{\Om_{2}}^{++}$ satisfies $(L_{\Om_2,m,\si}+\la)[\psi]\leq0$ in $\R\times\ol\O_2$. Define 
$$
\psi_{\ol{\O}_1}(t,x)=\psi(t,x),\quad (t,x)\in\R\times\ol{\O}_1\subset\R\times\ol{\Om}_{2}. 
$$
Then, $\psi_{\ol{\O}_1}\in\XX_{\Om_{1}}^{++}$. Moreover, for any $(t,x)\in\R\times\ol{\O}_1$
\begin{equation*}
\begin{split}
(L_{\O_1,m,\si}+\la)[\psi_{\ol{\O}_1}](t,x)&=-\partial_{t}\psi_{\ol{\O}_1}(t,x)+\frac{D}{\si^{m}}\left[\int_{\O_1}J_{\si}(x-y)\psi_{\ol{\O}_1}(t,y)dy-\psi_{\ol{\O}_1}(t,x)\right]\\
&\quad+a(t,x)\psi_{\ol{\O}_1}(t,x)+\la\psi_{\ol{\O}_1}(t,x)\\
&\leq-\psi_{t}(t,x)+\frac{D}{\si^{m}}\left[\int_{\Om_2}J_{\si}(x-y)\psi(t,y)dy-\psi(t,x)\right]\\
&\quad+a(t,x)\psi(t,x)+\la\psi(t,x)\\
&=(L_{\Om_2,m,\si}+\la)[\psi](t,x)\leq0.
\end{split}
\end{equation*}
That is, $(\la,\psi_{\ol{\O}_1})$ is a test pair for $\la_{p}(-L_{\O_1,m,\si})$, and hence, $\la\leq\la_{p}(-L_{\O_1,m,\si})$. Taking the supremum of all such possible $\la$, we arrive at $\la_{p}(-L_{\Omega_2,m,\si})\leq\la_{p}(-L_{\O_1,m,\si})$. The later statement can be proven by similar arguments as in the proof of \cite[Proposition 2 (iii)]{C1} under the help the equalities $\lambda_p(-L_{\Om_{i},m,\si})=\lambda_p'(-L_{\Om_{i},m,\si})$, $i=1,2$, which follows from Theorem \ref{thm-principal-vs-generalized}.

(5) It follows from classical perturbation theory of isolated eigenvalues (see e.g. \cite{Ka95}). In fact, for each fixed $\si_{0}>0$, we can write $L_{\Om,m,\si}$ as
$$
L_{\Om,m,\si}=L_{\Om,m,\si_{0}}+U_{\si,\si_{0}},
$$
where
$$
U_{\si,\si_{0}}[v](t,x)=\frac{D}{\si^{m}}\left[\int_{\Om}J_{\si}(x-y)v(t,y)dy-v(t,x)\right]-\frac{D}{\si_{0}^{m}}\left[\int_{\Om}J_{\si_{0}}(x-y)v(t,y)dy-v(t,x)\right].
$$
The result then follows from the facts that $U_{\si,\si_{0}}$ is bounded and linear, and $U_{\si,\si_{0}}\to0$ in norm as $\si\to\si_{0}$.
\end{proof}



Next, we prove Theorem \ref{thm-scaling-limit-eigenvalue}  concerning the scaling limits of principal eigenvalues.

%
%
%
%
\begin{proof}[{\bf Proof of Theorem \ref{thm-scaling-limit-eigenvalue}}]
(1) We here only prove the result in the case $m>0$; the $m=0$ case can be proven similarly. Thus, we assume $m>0$.

By Proposition \ref{prop-principal-e-variation-1}, we find $\la_{1}(-L_{\Om,m,\si})<\frac{D}{\si^{m}}-\max_{x\in\ol{\Om}}a_{T}(x)$, which implies 
$$
\limsup_{\si\to\infty}\la_{1}(-L_{\Om,m,\si})\leq-\max_{x\in\ol{\Om}}a_{T}(x).
$$
It remains to show that 
\begin{equation}\label{to-prove-03-11-17}
\liminf_{\si\to\infty}\la_{1}(-L_{\Om,m,\si})\geq-\max_{x\in\ol{\Om}}a_{T}(x).
\end{equation}
To do so, let us fix some constant $\phi_{0}>0$. One verifies that for each $x\in\ol{\Om}$, the function
\begin{equation}\label{special-test-function-000001}
\phi(t,x)=e^{\int_{0}^{t}[a(s,x)-a_{T}(x)]ds}\phi_{0},\quad t\in\R
\end{equation}
is a positive $T$-periodic solution of the ODE $v_{t}=a(t,x)v-a_{T}(x)v$. Clearly, $\phi\in\XX_{\Om}^{++}$ and we may choose $\phi_{0}$ such that $\sup_{\R\times\ol\O}\phi(t,x)=1$. For any $\ep>0$, we see that
\begin{equation}
\begin{split}
&(L_{\Om,m,\si}-\max_{z\in\ol\Om}a_{T}(z)-\ep)[\phi](t,x)\\
&\quad\quad=-\phi_{t}(t,x)+\frac{D}{\si^{m}}\left[\int_{\Om}J_{\si}(x-y)\phi(t,y)dy-\phi(t,x)\right]+[a(t,x)-\max_{z\in\ol\Om}a_{T}(z)-\ep]\phi(t,x)\\
&\quad\quad\leq\frac{D}{\si^{m}}\left[\int_{\Om}J_{\si}(x-y)\phi(t,y)dy-\phi(t,x)\right]-\ep\phi(t,x).
\end{split}\label{06.02.1}
\end{equation}
As $\min_{(t,x)\in\R\times\ol{\Om}}\phi(t,x)>0$ and
$$
\left\|\frac{D}{\si^{m}}\left[\int_{\Om}J_{\si}(\cdot-y)\phi(t,y)dy-\phi(t,\cdot)\right]\right\|_\infty\to0\quad\text{as}\quad\si\to\infty,
$$
there is $\si_{\ep}>0$ such that $(L_{\Om,m,\si}-\max_{z\in\ol\Om}a_{T}(z)-\ep)[\phi]\leq0$ for all $\si\geq\si_{\ep}$, which implies that
$$
\la_{1}(-L_{\Om,m,\si})=\la_{p}(-L_{\Om,m,\si})\geq-\max_{z\in\ol\Om}a_{T}(z)-\ep,\quad\si\geq\si_{\ep}.
$$
The arbitrariness of $\epsilon$ then yields \eqref{to-prove-03-11-17}. Hence, 
$$
\lim_{\si\to\infty}\la_{1}(-L_{\Om,m,\si})=-\max_{x\in\ol{\Om}}a_{T}(x).
$$

(2) Let $\phi$ be as in \eqref{special-test-function-000001} and  $\tilde{\phi}$ be a twice continuously differentiable, positive and $T$-periodic function on $\R\times\R^{N}$ such that $\tilde{\phi}(t,x)=\phi(t,x)$ for $(t,x)\in\R\times\ol\Om$. For any $\ep>0$, similar arguments as in \eqref{06.02.1} lead to
\begin{equation*} 
\begin{split}
&\left(L_{\Om,m,\si}-\max_{z\in\ol\Om}a_{T}(z)-\ep\right)[\phi](t,x)\\
&\quad\quad\leq \frac{D}{\si^{m}}\left[\int_{\Om}J_{\si}(x-y)\phi(t,y)dy-\phi(t,x)\right]-\ep\phi(t,x)\\
&\quad\quad\leq \frac{D}{\si^{m}}\left[\int_{\R^{N}}J_{\si}(x-y)\tilde{\phi}(t,y)dy-\tilde{\phi}(t,x)\right]-\ep\phi(t,x),\quad \quad(t,x)\in\R\times\ol{\Om}
\end{split}
\end{equation*}
It then follows from Taylor's expansion that
\begin{equation*}
\begin{split}
\frac{1}{\si^{m}}\left[\int_{\R^{N}}J_{\si}(x-y)\tilde{\phi}(t,y)dy-\tilde{\phi}(t,x)\right]&=\frac{1}{\si^{m}}\left[\int_{\R^{N}}J(z)\tilde{\phi}(t,x+\si z)dz-\tilde{\phi}(t,x)\right]\\
&=\si^{2-m}\int_{\R^{N}}J(z)\frac{z_{N}^{2}}{2}dz\De\tilde{\phi}(t,x)+O(\si^{\al})
\end{split}
\end{equation*}
for some $\al>0$. In particular, 
$$
\frac{D}{\si^{m}}\left[\int_{\R^{N}}J_{\si}(x-y)\tilde{\phi}(t,y)dy-\tilde{\phi}(t,x)\right]\to0\quad \text{as}\quad\si\to0\quad\text{uniformly in}\,\,(t,x)\in\R\times\ol{\Om}.
$$
As $\min_{(t,x)\in\R\times\ol{\Om}}\phi(t,x)>0$, for any $\epsilon>0$, there exists $\si_{\ep}>0$ such that   
$$
(L_{\Om,m,\si}-\max_{z\in\ol\Om}a_{T}(z)-\ep)[\phi]\leq0,\quad (t,x)\in\R\times\ol{\Om}
$$ 
for all $\si\geq\si_{\ep}$. By the definition of $\lambda_p(-L_{\Om,m,\si})$ and Proposition \ref{prop-principal-e-variation-1}(2), we obtain
\begin{equation}\label{17.5.1}
\liminf_{\si\to0^{+}}\la_{1}(-L_{\Om,m,\si})=\liminf_{\si\to0^{+}}\la_{p}(-L_{\Om,m,\si})\geq-\max_{x\in\ol\Om}a_{T}(x).
\end{equation}

We show the reverse inequality, namely, 
\begin{equation}\label{17.5.1-03-12-17}
\limsup_{\si\to0^{+}}\la_{1}(-L_{\Om,m,\si})\leq-\max_{x\in\ol\Om}a_{T}(x).
\end{equation}
For any $\ep>0$, there exists an open ball $B_\epsilon\subset\Om$ of radius $\ep$ such that 
$$
a_T(x)+\epsilon>\max_{x\in\ol\O} a_T(x),\quad x\in B_\epsilon. 
$$
Let $\tilde{\phi}_\epsilon:\R\times\R^{N}\to[0,\infty)$ be a continuous and $T$-periodic function and satisfies 
\begin{equation*}
\begin{split}
\tilde{\phi}_\epsilon(t,x)&=\phi(t,x),\quad(t,x)\in\R\times\ol B_\epsilon,\\ \tilde{\phi}_\epsilon(t,x)&=0,\quad \quad(t,x)\in \R\times(\R^{N}\bs B_{2\epsilon})\quad\text{and}\\ \sup_{\R\times\R^N}\tilde{\phi}_\epsilon(t,x)&\leq \sup_{\R\times\R^N}{\phi}_\epsilon(t,x)=1. 
\end{split}
\end{equation*}
Obviously, $\tilde{\phi}_\epsilon$ is also in $C^2(t,\cdot)$ for each $t\in\R$. We see
\begin{equation*}
\begin{split}
\mathcal{I}_{\sigma,m}&=\frac{1}{\sigma^m}\int_{\R^{N}}J_{\si}(x-y)(\tilde{\phi}_\epsilon(t,y)-\tilde{\phi}_\epsilon(t,x))dy\\
&=\frac{1}{\sigma^m}\int_{\R^N}J(z)(\tilde\phi_\epsilon(t,x+\sigma z)-\tilde{\phi}_\epsilon(t,x))dz\\
&=\sigma^{2-m}\int_{\R^N}J(z)\frac{z_N^2}{2} dz\Delta\tilde{\phi}_\epsilon(t,x)+O(\sigma^\alpha).
\end{split}
\end{equation*}
By assumptions on $J$, there holds $\int_{\R^N}J(z)\frac{z_N^2}{2} dz<\infty$. We then deduce that, for $m\in[0,2)$,  
$$
\mathcal{I}_{\sigma,m}\to0\,\,\quad\text{uniformly in}\,\, (t,x)\in\R\times\ol\O\quad\text{as} \quad\si\to0^{+}.
$$ 

Choosing $\epsilon=\sigma^k$, with $k=\frac{m+2N}{N}$, we have for each for $x\in B_\epsilon$,
\begin{equation*}
\begin{split}
&\left(L_{B_\epsilon,m,\si}-\max_{z\in\ol{\Om}}a_{T}(z)+\ep+\ln(\ep+1)\right)[\phi](t,x)\\
&\quad\quad=-\phi_{t}(t,x)+\frac{D}{\si^{m}}\left[\int_{B_\epsilon}J_{\si}(x-y)\phi(t,y)dy-\phi(t,x)\right]\\
&\quad\quad\quad+\left[a(t,x)-\max_{z\in\ol{\Om}}a_{T}(z)+\ep+\ln(\ep+1)\right]\phi(t,x)\\
&\quad\quad\geq\frac{D}{\si^{m}}\left[\int_{B_\epsilon}J_{\si}(x-y)\phi(t,y)dy-\phi(t,x)\right]+\ln(\ep+1)\phi(t,x).\\
&\quad\quad =\frac{D}{\sigma^m}\left[\int_{\R^N}J_{\si}(x-y)\tilde{\phi}_\epsilon(t,y)dy- \tilde{\phi}_\epsilon(t,x)- \int_{B_{2\epsilon}\setminus B_\epsilon}J_{\si}(x-y)\tilde{\phi}_\epsilon(t,y)dy\right]\\
&\quad\quad\quad+\ln(\ep+1)\phi(t,x)\\
&\quad\quad =\frac{D}{\sigma^{m-2}}\int_{\R^N}J(z)\frac{z_N^2}{2} dz\Delta\tilde{\phi}(t,x)+O(\sigma^\alpha)-\frac{D}{\sigma^{m+N}}\int_{B_{2\epsilon}\setminus B_\epsilon}J(x-y)\tilde{\phi}_\epsilon(t,y)dy\\
&\quad\quad\quad+\ln(\ep+1)\phi(t,x)\\
&\quad\quad=\mathcal{I}_{\sigma,m}^1+\mathcal{I}_{\sigma,m}^2+\mathcal{I}_{\sigma,m}^3 +O(\sigma^\alpha),
\end{split}
\end{equation*}
where
\begin{equation*}
\begin{split}
\mathcal{I}_{\sigma,m}^1&=\frac{D}{\sigma^{m-2}}\int_{\R^N}J(z)\frac{z_N^2}{2} dz\Delta\tilde{\phi}(t,x)+O(\sigma^\alpha),\\
\mathcal{I}_{\sigma,m}^2&=-\frac{D}{\sigma^{m+N}}\int_{B_{2\epsilon}\setminus B_\epsilon}J(x-y)\tilde{\phi}_\epsilon(t,y)dy,\\
\mathcal{I}_{\sigma,m}^3&=\ln(\ep+1)\phi(t,x).
\end{split}
\end{equation*}

Since $\min_{(t,x)\in\R\times\ol{\Om}}\phi(t,x)>0$, one has $\min_{(t,x)\in\R\times\ol B_\epsilon}\phi(t,x)>0$ uniformly in $\epsilon$. Clearly, the following estimates hold:
$$|\mathcal{I}_{\sigma,m}^1|\leq C_1\sigma^{2-m},\quad|\mathcal{I}_{\sigma,m}^2|\leq \frac{C_2\sigma^{kN}}{\sigma^{m+N}}\quad\text{and}\quad |\mathcal{I}_{\sigma,m}^3|\geq C_3\ln (\sigma^k+1).$$
For any $\alpha,\beta>0$, one has 
$$\lim_{\sigma\to0}\frac{\sigma^\alpha}{\ln(\sigma^\beta+1)}=0.$$
Hence,  for $0<\sigma\ll1$, there holds
$$
\left(L_{\Om,m,\si}-\max_{z\in\ol\Om}a_{T}(z)+\sigma^k+\ln(\sigma^k+1)\right)[\phi]\geq0\quad\text{in}\quad \R\times B_{\sigma^k}.
$$  
It then follows from the definition of $\lambda_p'(-L_{\Om,m,\si})$ and Proposition \ref{prop-principal-e-variation-1}(2) that
$$
\la_{1}(-L_{B_{\sigma^k},m,\si})=\la_{p}'(-L_{B_{\sigma^k},m,\si})\leq -\max_{x\in\ol\O} a_T(x)+\sigma^k+\ln(\sigma^k+1).
$$ 
By Proposition \ref{prop-principal-e-variation-1}(2)(4),
\begin{equation*}\label{principal-compare-01010}
\la_{1}(-L_{\Omega,m,\si})\leq\la_{1}(-L_{B_{\sigma^k},m,\si}), 
\end{equation*}
which yields 
$$
\la_{1}(-L_{\Omega,m,\si})\leq  -\max_{x\in\ol\O} a_T(x)+\sigma^k+\ln(\sigma^k+1) 
$$
for $0<\sigma\ll1$. This proves \eqref{17.5.1-03-12-17}.

Combining \eqref{17.5.1} and \eqref{17.5.1-03-12-17}, we conclude the expected limit
$$
\lim_{\sigma\to0^{+}}\la_{1}(-L_{\Omega,m,\si})=  -\max_{x\in\ol\O} a_T(x).
$$

(3) For $\sigma_1\geq\sigma_2$, we will show $\la_{1}(-L_{\Om,0,\si_1})\geq \la_{1}(-L_{\Om,0,\si_2})$. It is equivalent to show 
$$
\la_{p}(-L_{\Om,0,\si_1})\geq \la_{p}(-L_{\Om,0,\si_2}).
$$ 
Set $\O_\sigma=\frac{1}{\sigma}\O$, $a_\sigma(\omega)=a(\sigma\omega)$ for $\omega\in\Omega_\sigma$. One sees that
$$\la_{p}(-L_{\Om,0,\si})=\lambda_p(-\M_{\O_\sigma,0,1}-a_\sigma).$$
Therefore, we need to show that 
$$
\lambda_p(-\M_{\O_{\sigma_1},0,1}-a_{\sigma_1})\geq \lambda_p(-\M_{\O_{\sigma_2},0,1}-a_{\sigma_2}).
$$ 
To do so, it suffices to prove the inequality
$$
\lambda_p(-\M_{\O_{\sigma_1},0,1}-a_{\sigma_1})\geq\lambda,
$$
for any $\lambda<\lambda_p(-\M_{\O_{\sigma_2},0,1}-a_{\sigma_2})$.
 
Fix such a $\lambda$. By Proposition \ref{prop-principal-e-variation-1}, there exists  a positive function $\phi(t,x)\in \XX^{++}_{\O_2}$ such that
$$\M_{\O_{\sigma_2},0,1}[\phi]+(a_{\sigma_2}(t,x)+\lambda)[\phi]\leq0 $$
Since $\O$ contains the origin, one has $\O_{\sigma_1}\subset\O_{\sigma_2}$. Moreover $a_{\sigma_1}(t,x)\leq a_{\sigma_2}(t,x)$ for all $(t,x)\in \R\times\O_{\sigma_1}$. An easy computation yields
$$\M_{\O_{\sigma_1},0,1}[\phi]+(a_{\sigma_1}(t,x)+\lambda)[\phi]\leq \M_{\O_{\sigma_2},0,1}[\phi]+(a_{\sigma_2}(t,x)+\lambda)[\phi]\leq0\quad (t,x)\in \R\times\O_{\sigma_1}. $$
This implies $\lambda_p(-\M_{\O_{\sigma_1},0,1}-a_{\sigma_1})\geq\lambda.$

The proof is complete.
\end{proof}

Finally, we study the scaling limits of the positive $T$-periodic solution of \eqref{main-eqn-R} and we prove Theorem \ref{thm-scaling-limits-introduction}. We need the following lemma.

\begin{lemma}\label{lem-sol-limit-eqn}
Suppose {\bf(H2)}. If $\min_{x\in\ol{\Om}}a_{T}(x)>0$, then for each $x\in\ol{\Om}$, the equation $$v_{t}=f(t,x,v)$$ has a unique positive and $T$-periodic solution, denoted by $v^{*}(t,x)$, that is continuous in $x$.
\end{lemma}
\begin{proof}
It is known (see e.g. \cite{Col79}) that for each $x\in\ol{\Om}$, the equation $v_{t}=f(t,x,v)$ has a (unique) positive and $T$-periodic solution if and only if $a_{T}(x)>0$.
\end{proof}

We prove the following two theorems that consist of Theorem \ref{thm-scaling-limits-introduction}.

\begin{theorem}\label{thm-scaling-limit-positive-sol}
Suppose {\bf(H1)}, {\bf(H2)} and \eqref{newcond1}. Let $m\in[0,2)$.
\begin{enumerate}
\item If $\max_{x\in\ol{\Om}}a_{T}(x)>0$, then there exists $\si_{1}>0$ such that for each $\si\in(0,\si_{1}]$, equation \eqref{main-eqn-R} admits a unique positive solution $u^{*}_{\si}\in\XX^{++}_{\Om}$ that is globally asymptotically stable.

\item If $\min_{x\in\ol{\Om}}a_{T}(x)>0$,  then the limit
$$
\lim_{\si\to0^{+}}u^{*}_{\si}(t,x)=v^{*}(t,x)\quad\text{uniformly in}\quad (t,x)\in\R\times\ol{\Om}
$$
holds, where $v^{*}$ is given in Lemma \ref{lem-sol-limit-eqn}.
\end{enumerate}
\end{theorem}

\begin{theorem}\label{thm-scaling-limit-ps-infty}
Suppose {\bf(H1)}, {\bf(H2)} and \eqref{newcond1}. Let $m>0$.
\begin{enumerate}
\item If $\max_{x\in\ol{\Om}}a_{T}(x)>0$, then there exists $\si_{2}>0$ such that for each $\si\in[\si_{2},\infty)$, equation \eqref{main-eqn-R} admits a unique positive solution $u^{*}_{\si}\in\XX^{++}_{\Om}$ that is globally asymptotically stable.

\item If  $\min_{x\in\ol{\Om}}a_{T}(x)>0$,  then the limit
$$
\lim_{\si\to\infty}u^{*}_{\si}(t,x)=v^{*}(t,x)\quad\text{uniformly in}\quad(t,x)\in\R\times\ol{\Om}
$$
holds, where $v^{*}$ is given in Lemma \ref{lem-sol-limit-eqn}.
\end{enumerate}
\end{theorem}

Here, we only prove Theorem \ref{thm-scaling-limit-positive-sol}. The proof of Theorem \ref{thm-scaling-limit-ps-infty} can be done along the same line.

\begin{proof}[{\bf Proof of Theorem \ref{thm-scaling-limit-positive-sol}}]
(1) It follows from Theorem \ref{thm-scaling-limit-eigenvalue} and Theorem \ref{thm-persistence-criterion-b}.

(2) We claim that for each $0<\ep\ll1$, there exists $\si_{\ep}>0$ such that for each $\si\in(0,\si_{\ep})$ there holds
$$
v^{*}(t,x)-\ep\leq u^{*}_{\si}(t,x)\leq v^{*}(t,x)+\ep,\quad (t,x)\in\R\times\ol{\Om}.
$$
Let us prove the lower bound; the upper bound follows from similar arguments. Let $0<\ep\ll1$. Since $\min_{(t,x)\in\R\times\ol{\Om}}v^{*}(t,x)>0$, there exists $\de=\de(\ep)>0$ such that
$$
v(t,x):=(1-\de)v^{*}(t,x)\geq v^{*}(t,x)-\ep>0,\quad (t,x)\in\R\times\ol{\Om}.
$$
Note that
\begin{equation*}
\begin{split}
&-v_{t}(t,x)+\frac{1}{\si^m}\left[\int_{\Om}J_\sigma(x-y)v(t,y)dy-v(t,x)\right]+f(t,x,v(t,x))\\
&\quad\quad=-(1-\de)v^{*}_{t}(t,x)+\frac{1-\de}{\si^m}\left[\int_{\Om}J_\sigma(x-y)v^{*}(t,y)dy-v^{*}(t,x)\right]+(1-\de)f(t,x,v^{*}(t,x))\\
&\quad\quad\quad+\left[f(t,x,v(t,x))-(1-\de)f(t,x,v^{*}(t,x))\right]\\
&\quad\quad=\frac{1-\de}{\si^m}\left[\int_{\Om}J_\sigma(x-y)v^{*}(t,y)dy-v^{*}(t,x)\right]+\left[f(t,x,v(t,x))-(1-\de)f(t,x,v^{*}(t,x))\right].
\end{split}
\end{equation*}
We see that 
$$
\frac{1-\de}{\si^m}\left[\int_{\Om}J_\sigma(x-y)v^{*}(t,y)dy-v^{*}(t,x)\right]\to0\quad\text{as}\quad \si\to0^{+}\quad\text{uniformly in}\quad (t,x)\in\R\times\ol{\Om}.
$$
Since for each $(t,x)\in\R\times\ol{\Om}$
$$
f(t,x,v(t,x))-(1-\de)f(t,x,v^{*}(t,x))=v(t,x)\left[\frac{f(t,x,v(t,x))}{v(t,x)}-\frac{f(t,x,v^{*}(t,x))}{v^{*}(t,x)}\right]>0,
$$
where we used {\bf(H2)}-(3), there holds
$$
\inf_{(t,x)\in\R\times\ol{\Om}}\left[f(t,x,v(t,x))-(1-\de)f(t,x,v^{*}(t,x))\right]>0.
$$
Hence, there exists $\si_{\ep}>0$ such that for each $\si\in(0,\si_{\ep})$, there holds
\begin{equation}\label{sub-sol-limit-po-sol}
v_{t}(t,x)\leq\frac{1}{\si^m}\left[\int_{\Om}J_\sigma(x-y)v(t,y)dy-v(t,x)\right]+f(t,x,v(t,x)),\quad(t,x)\in\R\times\ol{\Om}.
\end{equation}

It remains to show that for each $\si\in(0,\si_{\ep})$, there holds $v(t,x)\leq u^{*}_{\si}(t,x)$ for all $(t,x)\in\R\times\ol{\Om}$. To do so, let us fix any $\si\in(0,\si_{\ep})$ and define
$$
\al_{*}=\inf\left\{\al>0:v(t,x)\leq\al u^{*}_{\si}(t,x)\,\,\text{for all}\,\,(t,x)\in\R\times\ol{\Om}\right\}.
$$
Since $\min_{(t,x)\in\R\times\ol{\Om}}u^{*}_{\si}(t,x)>0$ and $v(t,x)$ is bounded, $\al_{*}$ is well-defined. Due to the continuity of $v(t,x)$ and $u^{*}_{\si}(t,x)$, there holds
$$
 v(t,x)\leq \al_{*}u^{*}_{\si}(t,x),\quad (t,x)\in\R\times\ol{\Om}.
$$
Of course, they are not equal. Moreover, there exists $(t_{0},x_{0})\in\R\times\ol{\Om}$ such that
$$
v(t_{0},x_{0})=\al_{*}u^{*}_{\si}(t_{0},x_{0}).
$$

Clearly, if $\al_{*}\leq1$, then we are done. Therefore, let us assume $\al_{*}>1$. By \eqref{sub-sol-limit-po-sol} and the equation satisfied by $u^{*}_{\si}(t,x)$, we see that $w(t,x):=v(t,x)-\al_{*}u^{*}_{\si}(t,x)$ satisfies
\begin{equation*}
w_{t}(t,x)\leq\frac{1}{\si^m}\left[\int_{\Om}J_\sigma(x-y)w(t,y)dy-w(t,x)\right]+f(t,x,v(t,x))-\al_{*}f(t,x,u^{*}_{\si}(t,x)),\quad(t,x)\in\R\times\ol{\Om}.
\end{equation*}
However, since $w_{t}(t_{0},x_{0})=0$, $\int_{\Om}J_\sigma(x_{0}-y)w(t_{0},y)dy-w(t_{0},x_{0})<0$ and
$$
f(t_{0},x_{0},v(t_{0},x_{0}))-\al_{*}f(t_{0},x_{0},u^{*}_{\si}(t_{0},x_{0}))<f(t_{0},x_{0},v(t_{0},x_{0}))-f(t_{0},x_{0},\al_{*}u^{*}_{\si}(t_{0},x_{0}))=0,
$$
where we used $\al_{*}>1$ so that $\al_{*}u^{*}_{\si}(t_{0},x_{0})>u^{*}_{\si}(t_{0},x_{0})$, and hence, 
$$
\frac{f(t_{0},x_{0},u^{*}_{\si}(t_{0},x_{0}))}{u^{*}_{\si}(t_{0},x_{0})}>\frac{f(t_{0},x_{0},\al_{*}u^{*}_{\si}(t_{0},x_{0}))}{\al_{*}u^{*}_{\si}(t_{0},x_{0})}
$$ 
by {\bf(H2)}-(3), we arrive at
\begin{equation*}
w_{t}(t_{0},x_{0})>\frac{1}{\si^m}\left[\int_{\Om}J_\sigma(x_{0}-y)w(t_{0},y)dy-w(t_{0},x_{0})\right]+f(t_{0},x_{0},v(t_{0},x_{0}))-\al_{*}f(t_{0},x_{0},u^{*}_{\si}(t_{0},x_{0})),
\end{equation*}
which leads to a contradiction. Hence, $\al_{*}\leq1$ and the proof is completed.
\end{proof}

\begin{remark}
The asymptotic behavior of the positive $T$-periodic solution of the reaction-diffusion equation \eqref{main-eqn-rd} for small diffusion rates, i.e, as $d\to0^{+}$, has been studied by Daners and L\'opez-G\'omez in \cite{DLG}.
\end{remark}

\section{Maximum principle}\label{sec-mp}

We prove the maximum principle, namely, Theorem \ref{thm-mp-introduction}, for the operator $L_{\Om}$ defined in \eqref{main-eqn-linear}. 

\begin{proof}[{\bf Proof of Theorem \ref{thm-mp-introduction}}]
By Theorem \ref{thm-new-sufficient-cond-b}, $\la_{1}:=\la_{1}(-L_{\Om})$ is the principal eigenvalue of $-L_{\Om}$. Let $\phi\in \XX^{++}_{\Om}$ be an associated eigenfunction. Then,
\begin{equation}\label{eigen-equ-pv-010}
L_{\Om}[\phi]+\la_{1}\phi=0.
\end{equation}

We first prove the sufficiency, that is, $\la_1\geq0$ yields the maximum principle. To do so, let $u\in C^{1,0}([0,T]\times\ol\O)$ be nonzero and satisfy \eqref{mp-assumption}. Define
$$
w:=\frac{u}{\phi}.
$$
Then, simple calculations using \eqref{eigen-equ-pv-010} give
\begin{equation*}
0\geq L_{\Om}[u]=L_{\Om}[w\phi]=-w_{t}\phi+D\int_{\Om}J(x-y)\phi(t,y)[w(t,y)-w(t,x)]dy-\la_{1}w\phi.
\end{equation*}

Arguing by contradiction, let us assume that there exists $(t_{0},x_{0})\in(0,T]\times\Om$ such that $u(t_{0},x_{0})=\min_{[0,T]\times\Om}u\leq0$. Then, there exists $(t_{1},x_{1})\in(0,T]\times\Om$ such that $w(t_{1},x_{1})=\min_{[0,T]\times\Om}w\leq0$. It then follows 
\begin{equation*}
\begin{split}
-w_{t}(t_{1},x_{1})\phi(t_{1},x_{1})&\geq0,\\
D\int_{\Om}J(x_{1}-y)\phi(t_{1},y)[w(t_{1},y)-w(t_{1},x_{1})]dy&>0\quad\text{and}\\
-\la_{1}w(t_{1},x_{1})\phi(t_{1},x_{1})&\geq0.
\end{split}
\end{equation*}
Hence, we conclude that $L_{\Om}[u](t_{1},x_{1})>0$, which leads to a contradiction. This proves the sufficiency.

Now, we prove the necessity, that is, the maximum principle implies $\la_{1}\geq0$. For contradiction, let us assume $\la_{1}<0$. Let $\Om_{0}\stst\Om$. The size of $\Om_{0}$ will be specified later. Let $\eta:\ol{\Om}\to[0,1]$ be continuous and satisfy
$$
\eta(x)=\begin{cases}
0,&\quad x\in\partial\Om,\\
1,&\quad x\in\Om_{0}.
\end{cases}
$$

We calculate
$$
L_{\Om}[\eta\phi](t,x)=D\int_{\Om}J(x-y)\phi(t,y)[\eta(y)-\eta(x)]dy-\la_{1}\eta(x)\phi(t,x).
$$
There are three cases.
\begin{enumerate}
\item If $x\in\Om_{0}$, we have
\begin{equation*}
\begin{split}
L_{\Om}[\eta\phi](t,x)&=D\int_{\Om\bs\Om_{0}}J(x-y)\phi(t,y)[\eta(y)-1]dy-\la_{1}\phi(t,x)\\
&\geq-D\|J\|_{\infty}\|\phi\|_{\infty}|\Om\bs\Om_{0}|-\la_{1}\inf_{[0,T]\times\Om}\phi(t,x)|.
\end{split}
\end{equation*}

\item If $x\in\Om\bs\Om_{0}$ and $\eta(x)\geq\frac{1}{2}$, we find
\begin{equation*}
\begin{split}
L_{\Om}[\eta\phi](t,x)&\geq D\int_{\{y:\eta(y)\leq\eta(x)\}}J(x-y)\phi(t,y)[\eta(y)-\eta(x)]dy-\frac{1}{2}\la_{1}\phi(t,x)\\
&\geq-D\int_{\{y:\eta(y)\leq\eta(x)\}}J(x-y)\phi(t,y)dy-\frac{\la_{1}}{2}\phi(t,x)\\
&\geq-D\|J\|_{\infty}\|\phi\|_{\infty}\Om\bs\Om_{0}|-\frac{\la_{1}}{2}\inf_{[0,T]\times\Om}\phi(t,x)|,
\end{split}
\end{equation*}
where we used the fact that $\{y:\eta(y)\leq\eta(x)\}\subset\Om\bs\Om_{0}$.

\item If $x\in\Om\bs\Om_{0}$ and $\eta(x)\leq\frac{1}{2}$, then
\begin{equation*}
\begin{split}
L_{\Om}[\eta\phi](t,x)&\geq D\int_{\Om_{0}}J(x-y)\phi(t,y)[\eta(y)-\eta(x)]dy+D\int_{\Om\bs\Om_{0}}J(x-y)\phi(t,y)[\eta(y)-\eta(x)]dy\\
&\geq\frac{D}{2}\int_{\Om_{0}}J(x-y)\phi(t,y)dy-D\int_{\Om\bs\Om_{0}}J(x-y)\phi(t,y)dy\\
&\geq D\left[\frac{1}{2}\inf_{(t,x)\in[0,T]\times(\Om\bs\Om_{0})}\int_{\Om_{0}}J(x-y)\phi(t,y)dy-\|J\|_{\infty}\|\phi\|_{\infty}|\Om\bs\Om_{0}|\right]
\end{split}
\end{equation*}
\end{enumerate}

As $J(0)>0$, it is easy to choose $\Om_{0}$, say sufficiently close to $\Om$, such that $L_{\Om}(\eta\phi)\geq0$. Now, we have the following:
\begin{equation*}
\begin{cases}
L_\O[-\eta\phi]\leq 0 & \textrm{in $(0,T]\times\O$},\\
 -\eta\phi=0 & \textrm{on $(0,T]\times\partial\O$},\\
(-\eta\phi)(0,\cdot)= (-\eta\phi)(T,\cdot),
\end{cases}
\end{equation*}
Then, applying the comparison principle, we conclude that $-\eta\phi>0$ in $[0,T]\times\Om$, which clearly is a contradiction.
\end{proof}





\end{document}